\documentclass[11pt, a4paper, twoside]{article}
\usepackage{amsmath, amssymb, amscd, amsthm, mathrsfs, url}
\usepackage{amscd, graphicx, caption, subcaption}
\usepackage{wrapfig}
\usepackage{array, enumerate, mathtools}
\usepackage{amssymb, ifpdf, xspace,lipsum}
\usepackage{fancyhdr, float}
\usepackage[left=.75in,top=.75in,right=.75in]{geometry}
\usepackage{hyperref}
\newtheorem{thm}{Theorem}
\theoremstyle{definition}
\newtheorem{defn}{Definition}
\theoremstyle{plain}
\newtheorem{cor}[thm]{Corollary}
\newtheorem{lemma}[thm]{Lemma}
\newtheorem{prop}[thm]{Proposition}
\newtheorem{conj}[thm]{Conjecture}
\theoremstyle{definition}

\newtheoremstyle{example}{\topsep}{\topsep}%
 {}%         Body font
 {}%         Indent amount (empty = no indent, \parindent = para indent)
 {\bfseries}% Thm head font
 {.}%        Punctuation after thm head
 { }%\newline}%     Space after thm head (\newline = linebreak)
 {}%\thmname{#1}}%         Thm head spec

\theoremstyle{example}
\newcommand{\ds}[1]{\displaystyle{#1}}

\newcommand{\N}{\mathbb{N}}

\newcommand{\Z}{\mathbb{Z}}

\newcommand{\cC}{\mathcal{C}}

\newcommand{\inedges}{\mathrm{in}}
\newcommand{\idf}{Id_{\mathrm{f}}}
\newcommand{\idr}{Id_{\mathrm{r}}}

\begin{document}
\title{The Abelian Sandpile Model on Fractal Graphs}
\author{Samantha Fairchild\thanks{Research supported by the National Science Foundation through the Research Experiences for Undergraduates Programs at Cornell, grant DMS-1156350}$\,\,^1$,
Ilse Haim\footnotemark[1]$\,\,^2$,
Rafael G.~Setra\footnotemark[1]$\,\,^3$,\\
Robert S.~Strichartz\footnote{Research supported in part by the National Science Foundation, grant DMS-1162045}$\,\,^4$, and Travis Westura$^5$}

\date{}

\maketitle

\vspace{-0.5cm}

\begin{center}
\begin{footnotesize}

$^1\!\!$
\textit{Mathematics Department, Houghton College, Houghton NY 14744}\\
\textit{Current: Mathematics Department, University of Washington, Seattle WA 98125}\\
\textit{samantha.fairchild15@houghton.edu, skayf@uw.edu}

\smallskip

$^2\!\!$
\textit{Department of Mathematics, University of Maryland, College Park MD 20742}\\
\textit{Current: Google Inc., 1600 Amphitheatre Parkway, Mountain View CA 94043}\\
\textit{ilsehaim11@gmail.com}

\smallskip

$^3\!\!$
\textit{Department of Electrical and Computer Engineering, University of Maryland, College Park MD 20742}\\
\textit{Current: Department of Electrical Engineering, Stanford University, Stanford CA 94305}\\
\textit{rafael.g.setra@gmail.com}

\smallskip

$^4\!\!$
\textit{Mathematics Department, Cornell University, Ithaca NY 14852}\\
\textit{str@math.cornell.edu}

\smallskip

$^5\!\!$
\textit{Mathematics Department, Cornell University, Ithaca NY 14852}\\
\textit{tsw52@cornell.edu}

\end{footnotesize}
\end{center}

\begin{abstract}
We study the Abelian sandpile model, a process where stacks of chips are placed on a graph's vertices.
When the number of chips on a vertex is at least the vertex's degree, one chip is distributed from that vertex to each neighboring vertex.
This model has been shown to form fractal patterns on the integer lattice, and using these fractal patterns as motivation, we consider the model on graph approximations of post critically finite (p.c.f) fractals.
We determine asymptotic behavior of the diameter of sites toppled, characterize graphs that exhibit a periodic number of chips with respect to the initial placement, and investigate properties of the Sandpile Group on these graphs.

\end{abstract}

\smallskip
\noindent \textbf{Keywords:} fractal graphs, Abelian sandpile, Sierpinski Gasket, growth model, identity of sandpile group
\newline
AMS Classifications: 05C25, 20F65, 91B62, 05C75

\section{Introduction}

The Abelian Sandpile Model arose as a representation of self-organized criticality, a concept popularized in a 1987 paper by Bak, Tang, and Wiesenfeld \cite{SOC}.
The model associates each vertex of a graph with a number representing the height of a stack of chips placed at that vertex (we use the term ``chips'' for consistency with the terminology ``chip-firing game'', but other terms, e.g.\ sand, grains, or particles, are used throughout the literature).
These sandpile models are called ``abelian'' because an abelian group structure is defined using these chip configurations.
For general reading on this subject, see \cite{Chapter2,CaraccioloASM,DharAlgebraicAspects,MeesterASM, Creutz,DeepakDhar,Antal}.

In this paper we study the asymptotic and periodic patterns produced in the abelian sandpile model when placing $n$~chips on a single vertex~$v_0$ in graph approximations of fractals.
The fractals we examine include the Sierpinski gasket ($SG$, \ref{fig:SG2}), the Hexagasket ($HG$, \ref{fig:hg2mini}), the Mitsubishi gasket ($MG$, \ref{fig:mg2mini}), and the Pentagasket ($PG$, \ref{fig:pg2mini}).
Fey, Levine, and Perez show in \cite{Fey} that starting with $n$~chips on the origin of~$\mathbb{Z}^d$, the diameter of the set of sites that topple has order~$n^{1/d}$.
We extend this idea to fractal approximations by replacing~$d$ with the fractal's Hausdorff dimension.
Furthermore, we show that, under certain conditions, the number of chips on a given vertex~$v$ is periodic with respect to the initial $n$~chips on~$v_0$.
We additionally present a conjectured closed-form formula for the order of the Sandpile Group of an $m$th-level graph approximation of the Sierpinski Gasket and analyze the patterns exhibited by the identity configuration of various fractals.

The rest of this paper is organized as follows: \ref{sec:prelim} provides basic definitions of the sandpile model and establishes the notation used throughout the paper; \ref{sec:bound-grow} examines boundary growth on the Sierpinski Gasket; \ref{sec:period} proves that chip configurations on various fractal graphs exhibit periodic behavior in the presence of certain symmetries; \ref{sec:group-structure} discuss the properties of the Sandpile group for our fractal graphs; and \ref{sec:identities} examines the Sandpile group's identity configuration.

%-------------------------------------------------------------------------------

\section{Preliminaries}\label{sec:prelim}

Let ${G = (V \cup \{s\}, E)}$ be a connected graph of vertices~$V$, edges~$E$, and a distinguished vertex~$s$ called the sink.
A configuration on $G$~is a function ${\eta : V \to \N}$, with the intuitive meaning of placing a stack of chips on each vertex.
A configuration is called stable if for each ${v\in V}$, ${\eta(v) < \deg(v)}$.
If a vertex is unstable, the vertex topples and sends one chip to each of its neighbors.
Chips that land on the sink vertex are removed from the graph.
If multiple vertices must topple, the final configuration is independent of the order in which the vertices topple.
We denote the stable configuration corresponding to~$\eta$ as~$\eta^{\circ}$ and define the odometer function ${u : V \to \N}$ as the number of times that a vertex~$v$ topples throughout the stabilization process.

\begin{wrapfigure}{o}{0.45\textwidth}
 \centering
	\includegraphics[width=0.43\textwidth]{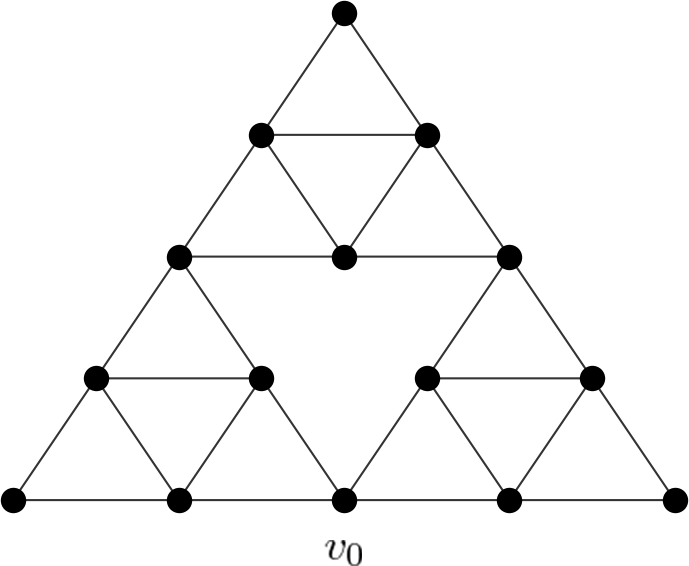}
	\caption{The second level graph approximation of $SG$: $SG_2$.}
	\label{fig:SG2}
	\vspace{-5pt}
\end{wrapfigure}

In \ref{sec:bound-grow} and \ref{sec:period} we consider the ``chip-firing game''.
Here we place chips on a central vertex~$v_0$ and analyze the stabile configuration resulting from toppling this stack of chips.
We use ${d(v_0, v)}$ to denote the distance of vertex~$v$ from~$v_0$ with respect to the graph metric, and we let ${V_{r} = \{ v\in V : d(v_0, v) \leq r\}}$ be the set of vertices at distance at most~$r$ from~$v_0$.

Let~$T_n$ be the set of vertices that topple when placing $n$~chips on an initial vertex~$v_0$.
More precisely, ${T_n = \{v \mid u(v) \ge 1\}}$.
These sets are nested, ${T_1 \subseteq T_2 \subseteq ...}$, and we consider them as a growth model where the number of chips~$n$ plays the role of a time parameter.
In this paper we restrict ourselves to the case where all vertices other than~$v_0$ begin with~$0$ chips---other work such as \cite{Fey} examines the graph~$\Z^d$ with other background conditions.

Given a graph~$G$, the Sandpile group of~$G$ consists of the set of stable configurations on~$G$ together with an operation~$\oplus$ given by
\begin{equation*}
	\eta_1 \oplus \eta_2 := {(\eta_1 + \eta_2)}^{\circ}.
\end{equation*}
Intuitively, this operation consists of stacking the chips of~$\eta_1$ and~$\eta_2$ on top of each other and then stabilizing the configuration.
Letting~$\Delta$ denote the Laplacian matrix of~$G$, we define the reduced Laplacian~$\Delta'$ by removing the row and column of~$\Delta$ that correspond to the sink vertex.
Writing $\Z^{|V|-1}\Delta'$ as the integer row span of~$\Delta'$, the Sandpile group is isomorphic to~${\Z^{|V| - 1} / \Z^{|V| - 1} \Delta'}$.
The determinant of~$\Delta'$ equals the order of the Sandpile group.
This determinant also equals the number of spanning trees of~$G$.

The elements of the Sandpile group are called recurrent configurations.
The name ``recurrent'' comes from considering the Markov process where a vertex is chosen uniformly at random, a chip is added to this vertex, and the resulting configuration is stabilized.
This process has a single class of recurrent configurations, which form the elements of the Sandpile group.
More precisely, a configuration~$\sigma$ is recurrent if the following hold:
\begin{itemize}
	\item $\sigma$ is stable.
	\item Given any configuration~$\eta$, there exists a configuration~$\eta'$ such that~${\sigma = \eta + \eta'}$.
\end{itemize}

\begin{wrapfigure}{o}{0.45\textwidth}
 \centering
	\vspace{-10pt}
	\includegraphics[width=0.38\textwidth]{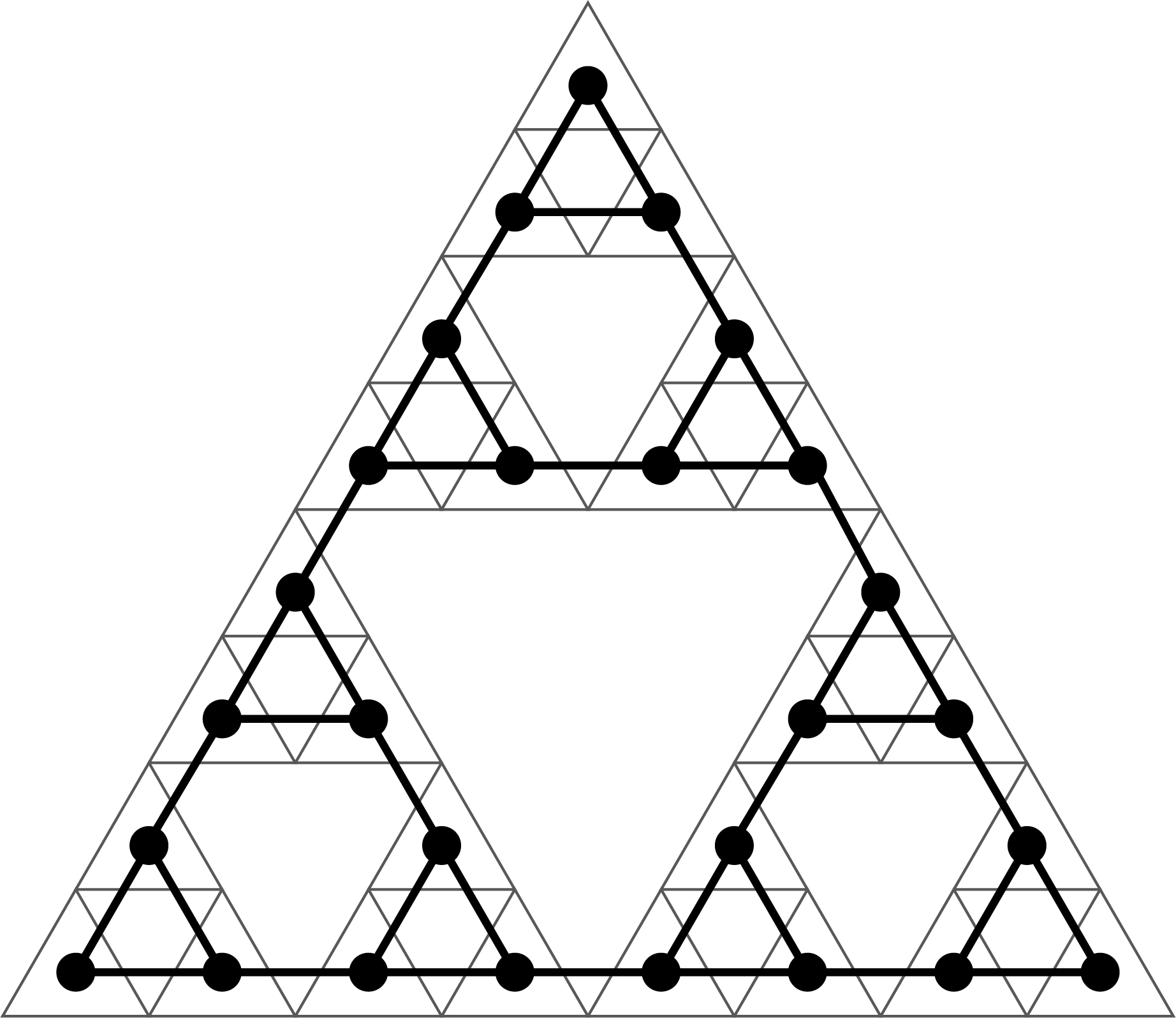}
	\caption{
		The Sierpinski Gasket Cell Graph, in which each level~$0$ cell is a node with edges to its neighbors.
	}
	\label{fig:SGC-outline}
	\vspace{-25pt}
\end{wrapfigure}

The primary graph we use throughout this paper is the Sierpinski Gasket, which we denote as~$SG$.
See~\cite{Strichartz} for an explicit construction of this graph.
We note that~$SG$ is an example of a post critically finite fractal, which is a class of fractals characterized by the ability to disconnect the fractal by removing a small set of junction vertices.
We use~$SG_m$ to represent an $m^{th}$~level graph approximation of the Sierpinski Gasket; as an example consider~$SG_2$ in \ref{fig:SG2}.
The level of graph approximation is defined analogously for the other fractals.
Additionally, we consider~$SGC$: the cell graph of~$SG$ (see \ref{fig:SGCell_Id}) constructed by placing a vertex in each cell of~$SG$ and connecting these vertices if their cells are connected in~$SG$.

%-------------------------------------------------------------------------------

\section{Boundary Growth}\label{sec:bound-grow}

Consider placing a stack of $n$~chips on vertex~$v_0$ of~$SG_m$ (see \ref{fig:SG2}).
Take~$m$ large enough so that chips never reach the boundary during the toppling procedure (that is, the chips remain contained in the bottom two level~${m - 1}$ subgraphs of~$SG_m$ and never reach the top triangle).

\begin{figure}
	\includegraphics[scale=0.40]{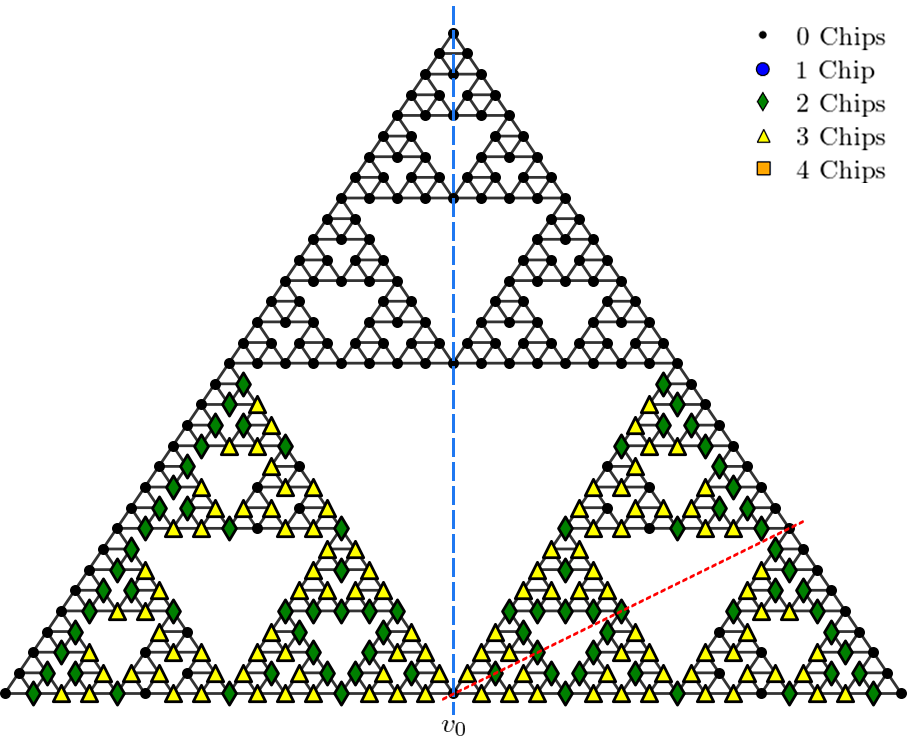}
	\centering
	\caption{
    Stabilized~$SG_5$ after placing $500$~chips on~$v_0$.
    Note two lines of symmetry through~$v_0$: one that bisects the entire graph and another that bisects the bottom-right subgraph.
  }
	\label{fig:sg5500}
\end{figure}

In \ref{fig:sg5500} we emphasize two lines of symmetry: one that divides the left and right portions of the graph and another that divides the bottom right portion of the graph.
Since chips from~$v_0$ never reach the top third of the graph, the only remaining connection between the bottom two halves is through~$v_0$.
Thus we concentrate our analysis on only the bottom right third of the graph, as the behavior there is symmetric to that of the bottom left third.

Our goal is to establish the growth rate of the diameter of~$T_n$, which is given by
\begin{equation*}
  \max_{u, v \,\in\, T_n} d(u, v).
\end{equation*}
Towards this end, we consider the set~$S_n$ of vertices that at some point have contained a chip during the stabilization process.
We have ${S_n = T_n \cup \partial T_n}$, where ${\partial T_n}$ denotes vertices adjacent to those in~$T_n$.
Note that $S_n$ has the same asymptotic growth rate as~$T_n$.

Let the radius~$r_n$ of~$S_n$ be the maximum value of~$d(v_0, v)$ over all~${v \in V}$ in~$S_n$.
We prove the following theorem:
\begin{thm} \label{thm:SGGrowth}
  Placing $n$~chips on~$v_0$, the radius~$r_n$ grows asymptotically as ${r_n = \Theta\left(n^{1/D_f}\right)}$, where ${D_f = \frac{\log 3}{\log 2}}$ is the fractal dimension of~$SG$.
\end{thm}

We begin the proof of \ref{thm:SGGrowth} by providing a counting argument for a lower bound on~$r_n$.

\begin{lemma} \label{lem:lower}
	Let $n$ be the number of chips placed on~$v_0$ and $r_n$ be the radius of~$S_n$.
	Then for~${n \ge 4}$, ${\left(\frac{n}{20}\right)^{\!1/D_f} < r_n}$.
	Thus ${r_n = \Omega(n^{1 / D_f})}$.
\end{lemma}

\begin{proof}
	We need ${n \ge 4}$ in order for $v_0$ to topple and for $r_n$ to be positive.
	The following argument holds for~${n \ge 4}$.

	The maximum number of chips on a stable vertex~$v$ is~${\deg(v) - 1}$.
	If we fix an~$r$ and place the maximum stable number of chips on all vertices of~$V_{r-1}$, by adding one additional chip to~$v_0$ we guarantee a toppling into a vertex of distance~$r$.
	Thus given~$r_n$, we have an upper bound on the number of chips originally stacked on~$v_0$,
  \begin{equation*}
  	n \leq 1 + \smashoperator{\sum_{v\in V_{r_n-1}}} (\deg(v) - 1) .
  \end{equation*}
  But ${\deg(v) = 4}$ for all vertices, so this inequality becomes
  \begin{equation}  \label{eq:Vest}
  	n \leq 1 + 3 \left|V_{r_n-1}\right|.
  \end{equation}
  There exists an integer~${k > 0}$ such that ${2^{k-1} < r_n \leq 2^k}$.
  And for~$k > 0$,
  \begin{equation} \label{eq:Rest}
  	|V_{2^k}| = 3^k + 2 < 3(3^k) = 3(2^k)^{D_f}.
  \end{equation}
  Note the leftmost equality follows from the number of vertices in $SG_k$ being given by ${\frac{1}{2}(3^{k+1}+3)}$.
	The subgraph~$V_{2^k}$ is formed by joining two $SG_{k-1}$ graphs at a boundary vertex.

  Combining these inequalities, we observe
  \begin{equation}
  	n < 4 \left|V_{2^k}\right| < 4 \big(3(2^k)^{D_f}\big) = 20 (2^{k-1})^{D_f} < 20 (r_n^{D_f}).
  \end{equation}
  Rearranging this inequality yields our desired result.
\end{proof}

Next we obtain an upper bound.
In the following lemmas we denote the junction point at distance $2^j$ as $v_{2^j}$ (see \ref{fig:induction}).
\begin{figure}
	\centering
	\includegraphics[scale=0.25]{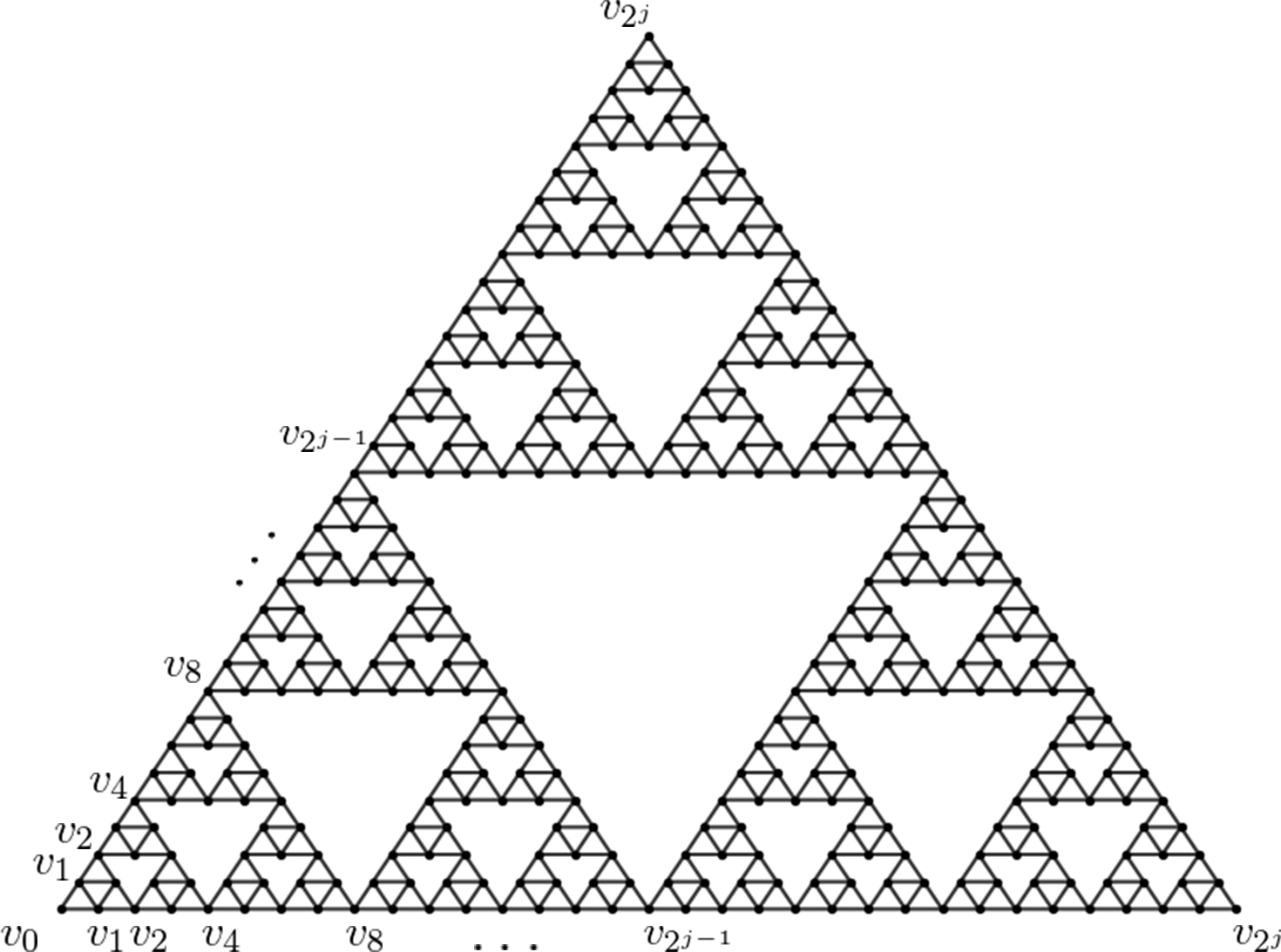}
	\caption{Vertex labels for $SG$ with maximum distance~$2^{j}$.}
	\label{fig:induction}
\end{figure}

\begin{lemma} \label{lem:pattern}
	If $T_n$ contains a vertex~$v$ with~${d(v_0, v) = 2^j}$, then for ${j \ge 1}$, ${V_{2^j - 1} \subseteq T_n}$.
\end{lemma}

\begin{proof}
	First consider a base case.
	By symmetry the vertices $v_1$ in \ref{fig:basecase} always have the same number of chips.
	The lowest~$n$ for which the $v_1$~vertices topple is~${n = 16}$.
	Once the $v_1$~vertices have $4$~chips, each will topple, and the final configuration will be as shown in \ref{fig:basecase1}.
	Here each $v_1$~vertex has~$2$ chips, and ${V_1 \subseteq T_n}$.

	\begin{figure}
		\centering
		\begin{subfigure}[b]{0.4\textwidth}
			\includegraphics[width = \textwidth]{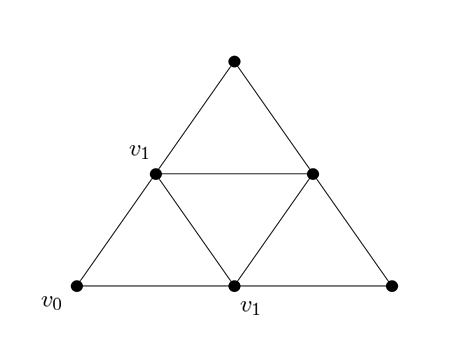}
			\caption{}
			\label{fig:basecase}
		\end{subfigure}
		\begin{subfigure}[b]{0.4\textwidth}
			\includegraphics[width = \textwidth]{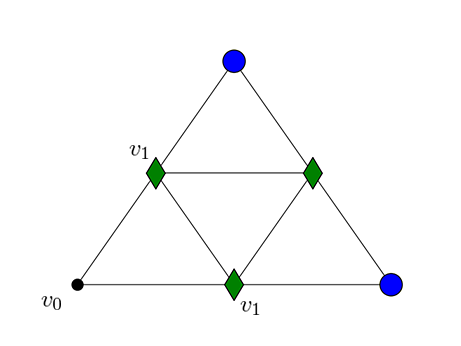}
			\caption{}
			\label{fig:basecase1}
		\end{subfigure}
		\caption{(a) Vertex labels for the base case: the right half of $SG$ with maximum distance 2. (b) Resulting configuration after each $v_1$ topples once.}
		\label{fig:period-lemma}
	\end{figure}

	Next assume the lemma holds for some~$j$.
	Noting that $SG$ is a p.c.f.~fractal, the only way vertices at distances larger than~$2^j$ receive chips is for both~$v_{2^j}$ to topple.
	Attached to each~$v_{2^j}$ is a gasket of depth~$2^j$.
	These gaskets are connected only at a single vertex, which has distance~$2^{j+1}$.
	Since the topplings may be carried out in any order, we can use the same sequence of topplings on these next two gaskets that was used in the first gasket.
	After these topplings are finished, the remaining chips may be stabilized in any order.
	By carrying out this sequence of topplings, we ensure that if a $v_{2^{j+1}}$~vertex topples, then all vertices in~$V_{2^{j+1}-1}$ also topple.
	Thus ${V_{2^{j+1}-1} \subseteq T_n}$.

\end{proof}

\begin{lemma}\label{finallem}
	If there are chips on vertices~$v$ with~${d(v_0, v) > 2^j}$, then every vertex in~$V_{2^j}$ must have toppled.
	That is, ${r_n > 2^j}$ implies~${V_{2^j} \subseteq T_n}$.
\end{lemma}
\begin{proof}
	For vertices at distance greater than~$2^j$ to have chips, the junction vertices~$2^j$ must topple.
  \ref{lem:pattern} then implies that ${V_{2^j - 1} \subseteq T_n}$, so in total ${V_{2^j} \subseteq T_n}$.
\end{proof}

The last idea needed is a theorem from Rossin \cite{Rossin}:
\begin{thm}(Rossin)
	Given a graph with vertices and edges~${G = (V,E)}$ and ${X \subseteq V}$ a set of connected vertices, define ${C_G(X) = \{(i,j)\in E \mid i \in X, j\in V\setminus X\}}$.
	The minimum number of chips that must be placed on~$X$ so that all vertices of~$X$ topple at least once, $M_X$, is ${|\inedges(X)| + |C_G(X)|}$, where $\inedges(X)$ is the set of edges inside~$X$.
\end{thm}

With this theorem, we establish an upper bound on~$r_n$.
\begin{lemma} \label{lem:upper}
	$\ds{r_n < \left(\frac{n}{2}\right)^{\!1/D_f}}$\!, and thus ${r_n = O(n^{1/D_f})}$.
\end{lemma}

\begin{proof}
	Note that ${\inedges(SG_k) = 3^{k+1}}$ and ${\inedges(V_{2^k}) = 2\cdot3^k}$.
	Rossin's theorem yields ${M_{V_{2^k}} > 2\cdot3^k}$.
	By \ref{finallem}, if ${r_n \ge 2^k}$, then ${V_{2^k} \subseteq T_n}$.
	Since $M_{V_{2^k}}$ is the minimum number of chips to place on $V_{2^k}$ such that all of its vertices topple, we must have ${n > M_{V_{2^k}}}$.
	Thus
	\begin{equation}
		n \geq M_{V_{2^k}} > 2 \cdot 3^k =2 (2^k)^{D_f} \geq 2{r_n}^{D_f}
	\end{equation}
	Rearranging this inequality yields the desired result.
\end{proof}

Using \ref{lem:lower} and \ref{lem:upper} respectively as lower and upper bounds, we obtain \ref{thm:SGGrowth}.
For reference, these bounds are shown in \ref{fig:SGGrowth_data}.

\begin{figure}[h]
	\centering
	\includegraphics[scale=0.7]{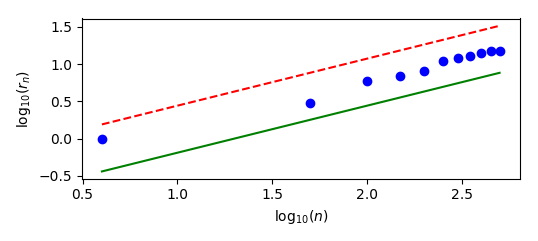}
	\caption{
		Resulting upper (red dashed) and lower (green solid) bounds of growth for the boundary on~$SG$ with respect to the number of chips, compared to actual data (circles). Plotted on a log-log scale with slopes approximately $\frac{1}{D_f}$.
	}
	\label{fig:SGGrowth_data}
\end{figure}

Now we relate the radius~$r_n$ of~$S_n$ to the diameter of~$T_n$.
\begin{cor} \label{cor:diameter}
	The diameter of the set of toppled sites~$T_n$ grows asymptotically as~${\Theta\left(n^{1/D_f}\right)}$.
\end{cor}
\begin{proof}
	Due to the symmetry of~$SG$ through~$v_0$, the diameter of~$S_n$ is equal to twice the radius~$r_n$, and thus has the same asymptotic behavior.
	As ${S_n = T_n \cup \partial T_n}$, the asymptotic growth rate of~$T_n$ is the same as that of~$S_n$.
\end{proof}

We note this property is not limited only to~$SG$.
By a similar argument to \ref{lem:pattern} and \ref{finallem}, we can show the following proposition.
\begin{prop}
	Let~$v_0$ be a vertex such that the junction points with the next level graph approximation are all at the same distance from~$v_0$.
	We then have the following results:
	\begin{itemize}
		\item $HG$:
			\begin{equation*}
				r_n = \Theta\left( n^{1/D_f}\right),
			\end{equation*}
			where $D_f = \frac{\log 6}{\log 3}$ is the fractal dimension of $HG$.
		\item $MG$:
			\begin{equation*}
				r_n = \Theta\left( n^{1/D_f}\right),
			\end{equation*}
			where $D_f = \frac{\log 6}{\log 2}$ is the fractal dimension of $MG$.
		\item $PG$:
			\begin{equation*}
					r_n = \Theta\left( n^{1/D_r}\right),
			\end{equation*}
			where ${D_f = \frac{\log 5}{\log \frac{3 + \sqrt{5}}{2}} \approx 1.67}$ is the fractal dimension of~$PG$, and ${D_r = \frac{\log 5}{\log(1+\sqrt{3})} \approx 1.60}$.
	\end{itemize}
\end{prop}

For explicit constructions of these graphs, see \cite{Strichartz} examples~$4.1.1$, $4.1.2$, and~$4.1.4$.
Note that~$PG$ is different from the other graphs---the fractal dimension is calculated with respect to the Euclidean metric, but we make use of the graph metric in the model.

\begin{figure}
	\centering
	\begin{subfigure}[b]{0.35\textwidth}
		\includegraphics[width = \textwidth]{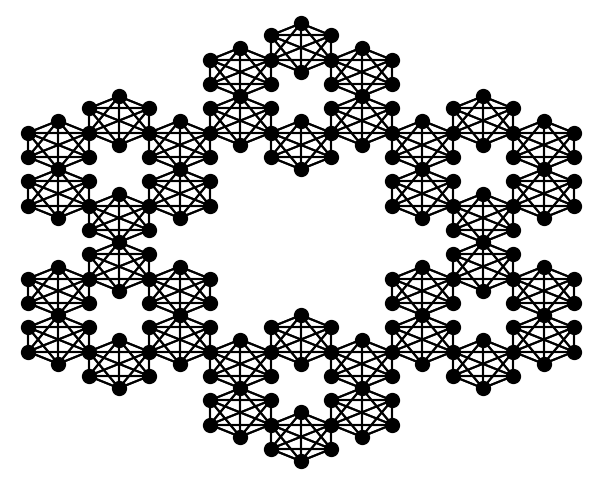}
		\caption{}
		\label{fig:hg2mini}
	\end{subfigure}
	\begin{subfigure}[b]{0.3\textwidth}
		\includegraphics[width = \textwidth]{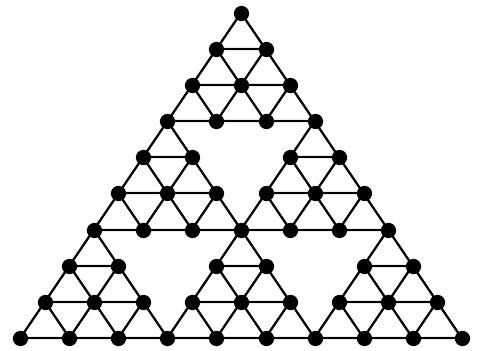}
		\caption{}
		\label{fig:mg2mini}
	\end{subfigure}
	\begin{subfigure}[b]{0.3\textwidth}
		\includegraphics[width = \textwidth]{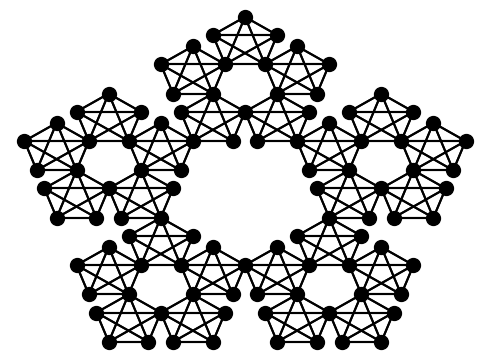}
		\caption{}
		\label{fig:pg2mini}
	\end{subfigure}
	\caption{(a) $HG_2$. (b) $MG_2$. (c) $PG_2$.}
	\label{fig:graphs}
\end{figure}

%-------------------------------------------------------------------------------

\section{Periodicity}\label{sec:period}

The number of chips on a vertex~$v$ of an $SG$~graph approximation repeats in a pattern that is periodic with respect to the number of chips~$n$ initially placed on~$v_0$.
That is, using $\eta_v^{\circ}(n)$ to denote the stabilized number of chips on vertex~$v$, each $\eta_v^{\circ}(n)$ is periodic with respect to~$n$.
An analogous property also holds for $HG$, $PG$, and~$MG$.
First we show the existence of this periodic behavior, then we further characterize its properties on subgraphs of fractal graph approximations.

In \ref{sec:prelim} we gave a definition of a graph~${G = (V \cup \{s\}, E)}$ with a single sink vertex.
It is also useful to consider a graph with a set~$B$ of multiple sink vertices~${G = (V \cup B, E)}$.
The original definition is recovered by collapsing this set of vertices into a single vertex.

We make use of the following properties of the sandpile model \cite{Chapter2}:
\begin{lemma} %This lemma is not directly cited in Chapter 2. However it is a direct result of the fact that the order of toppling does not matter. Thus we can add all the chips and then topple, or we can topple half way in between. The discussion of these concepts are stated in the first few pages of Chapter 2, so we cite it as a loose reference.
	For a graph~$A$ and configurations $\eta_1$, $\eta_2$, and $\eta_3$ on~$A$,
	\begin{equation*}
		\big((\eta_1+\eta_2)^\circ+\eta_3\big)^\circ = \big(\eta_1+\eta_2+\eta_3\big)^\circ.
	\end{equation*}
\end{lemma}

\begin{lemma}\label{recurrentguarantee}
	Let $A$ be a finite graph with a distinguished sink vertex and let $\eta$ be a nonzero configuration on $A$.
	Then there exists an integer $N(A,\eta)$ such that $(k\eta)^\circ$ is recurrent for $k \geq N(A,\eta)$.
\end{lemma}

\begin{defn}
	For a graph~${G = (V, E)}$ with~${B \subseteq V}$ as sinks, define the identify frame of the configuration~$\idf$ such that for~${v \in V}$, $\idf(v)$ is the number of edges that $v$ has connected to sinks.
\end{defn}

% This lemma is the Creutz identity which is equation 2.23 in Chapter 2
\begin{prop}\label{prop:Idf} (From \cite{DharAlgebraicAspects, Caracciolo}).
	Consider a graph~${G = (V, E)}$ with~${B \subseteq V}$ as sinks.
	A configuration $\eta_G$ is recurrent if and only if ${{(\idf +\eta_G)}^\circ = \eta_G}$, i.e.~$\idf$ acts as an identity configuration on recurrent configurations.
	Furthermore, in the stabilization process from~${\eta_G + Id_f}$ to~$\eta_G$, each vertex topples exactly once.
\end{prop}

Note that $\idf$ is not necessarily recurrent and thus may be different than the identity configuration~$\idr$ of the Sandpile group.

\subsection{Existence of Periodicity}

Consider a finite subset~${B \subseteq V}$ such that removing~$B$ and the edges connected to elements of~$B$ disconnects $G$ into two graphs, where at least one of these graphs is finite.
Denote this finite graph by~$F$ and define~$S$ to be the graph of~$F$ combined with vertices in~$B$ as sinks, including edges between elements of~$F$ and~$B$.
An example of this decomposition is in \ref{fig:periodDecomp}.

\begin{figure}[H]
	\centering
	\includegraphics[scale=0.2]{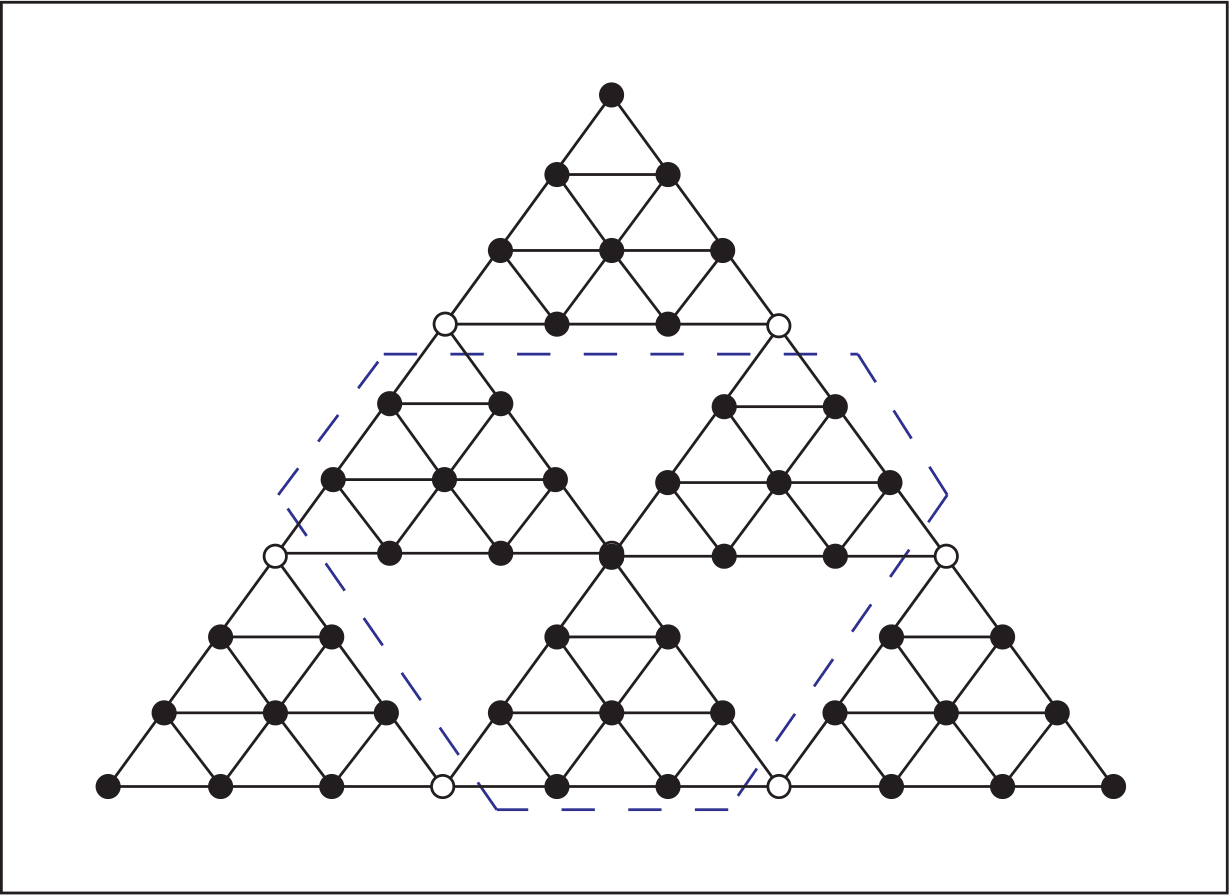}
	\caption{
		Example of a graph~$G$ where the white vertices are the subset~$B$.
		The graph~$F$ is contained in the central hexagonal outline.
		The graph~$S$ is the graph~$F$ with the white vertices from~$B$ as sinks.
	}
	\label{fig:periodDecomp}
\end{figure}

In the following, for~$B$ a subgraph of~$A$ with~$\eta_A$ a configuration on~$A$, we use ${\eta_A\big|_{B}}$ to denote the restriction of~$\eta_A$ to~$B$.
\begin{thm} \label{thm:invarianceConf}
	Let $G$, $F$, $B$, and~$S$ be as defined in the preceding paragraphs and let~$\eta_F$ and~$\eta_G$ be configurations on~$F$ and~$G$, respectively.
	Let ${\eta = k \eta_F + \eta_G}$ be a configuration on~$G$, where $\eta_F$ is set to~$0$ on vertices in~${G \setminus F}$.

	If stabilizing $\eta$ for ${k \geq N(F, \eta_F)}$ is possible and causes each vertex of~$B$ to topple the same number of times (i.e.~the odometer function is constant on~$B$), then ${\eta\big|_G}^{\circ}$ and ${\eta\big|_S}^{\circ}$ agree on~$F$.
	We write~${\eta\big|_G}$ in place of~$\eta$ to emphasize its difference from~${\eta\big|_S}$.
\end{thm}
To relate this theorem to \ref{sec:bound-grow}, the boundary~$B$ is a set of junction points~$v_{2^j}$, the inner graph~$F$ is the set of vertices~$V_{2^j - 1}$, the configuration~$\eta_F$ is a single chip on~$v_0$, and the background configuration~$\eta_G$ is~$0$ for all vertices.

\begin{proof}
	First consider $\eta\big|_S$.
	Because the vertices in~$B$ are the sinks of this graph, we stabilize the configuration by toppling only vertices in~$F$.
	This resulting configuration~${\eta\big|_S}^{\circ}$ is recurrent by \ref{recurrentguarantee}.

	Now consider $\eta\big|_G$.
	Topple only vertices in~$F$ until every vertex in~$F$ is stable.
	Now the configuration agrees with ${{\eta\big|_S}^{\circ}}$ on~$F$, as all the same topplings occurred.
	Subsequently, each time we topple once all of the vertices in~$B$ and then topple each vertex in~$F$ until all of~$F$ is stable, we still have $\eta_S$.
	That is, letting $\idf$ be the identity frame for~$B$, ${{\big(\idf + \eta\big|_S^{\circ}\big)}^{\circ} = {\eta\big|_S}^{\circ}}$ by \ref{prop:Idf}.

	Since the only vertices in~${G \setminus F}$ connected to~$F$ are the vertices in~$B$, we can stabilize~$\eta$ on all of~$G$ by repeatedly stabilizing it on~$F$ whenever~$F$ has an unstable vertex, then toppling all of the vertices on~$B$ once.
	The vertices outside of~$B$ may be stabilized arbitrarily.
	Hence ${\eta\big|_G}^{\circ}$ and ${\eta\big|_S}^{\circ}$ agree on~$F$.
\end{proof}

\begin{cor}
	For a configuration~$\eta$ that satisfies \ref{thm:invarianceConf}, vertices in~$F$ are eventually periodic with respect to the addition and stabilization of~$\eta_F$.
	That is, letting~${f_v(k)}$ denote~${(k\eta_F \oplus \eta_G)(v)}$ for each~${v \in F}$, we have that each~$f_v(k)$ is eventually periodic.
\end{cor}

\begin{proof}
	Let ${k \ge N(F, \eta_F)}$ and ${\eta = k\eta_F + \eta_G}$ for $\eta_F$ and $\eta_G$ configurations such that the hypotheses of \ref{thm:invarianceConf} are satisfied.
	By this theorem we have that ${\eta\big|_G}^{\circ}$ is equivalent to~${\eta\big|_S}^{\circ}$ on~$F$.
	However, $S$ is a finite graph and thus enters a cycle of configurations with respect to the repeated addition and stabilization of~$\eta_F$.
	Hence for all~${v \in F}$, $f_v$ is eventually periodic.
\end{proof}

This result applies to our fractal graph approximations, as we add chips only to a center point equidistant from all junction points.
If we use the junction points of the graphs as the set~$B$, we satisfy the hypotheses of \ref{thm:invarianceConf}.

%----------------------------------------------------------------------

\subsection{Periodicity of Nested Graphs}

We now consider a sequence of finite subsets~${B_n \subseteq V}$ such that removing each~$B_n$ individually disconnects $G$ into two graphs, as before.
Denote by~$F_n$ the finite connected graph resulting from the removal of~$B_n$, where the~$F_n$ form an increasing sequence~${F_1 \subseteq F_2 \subseteq \cdots}$.
Finally, define~$S_n$ to be the graph of~$F_n$ combined with the vertices in~$B_n$ as sinks; include edges in~$E_n$ that connect elements of~$F_n$ and~$B_n$.
We now fix a configuration~$\eta$ on~$F_1$.

We consider the set of configurations formed over~${k \in \N}$ by taking~$(k\eta)^\circ$.
For example, a sequence of chip stacks of increasing heights on a vertex~$v_0$ occurs when setting~$\eta$ to be the configuration with $1$~chip on~$v_0$ and $0$~chips elsewhere.
We take~$G$ and~$B_n$ as before such that given any~${k \in \N}$, the odometer function on~$B_n$ equals some constant~${c_k}$ in the stabilization of~$k\eta$ to~$(k\eta)^\circ$.
That is, ${u(v) = c_k}$ for all~${v \in B_n}$.
Letting $\idf^n$ denote the configuration
\begin{equation*}
	\underbrace{\idf \oplus \idf \oplus \cdots \oplus \idf}_{n \text{ times}},
\end{equation*}
note that $\idf^n$ is the same as configuration as we would obtain by toppling each vertex in~$B_n$ once with an otherwise $0$~background.

\begin{defn}\label{def:recurrent-config-cycle-set}
	For a configuration~$\eta$ defined on a graph~$G$ that is partitioned into subgraphs~$S_n$ as described previously, let $\cC_n$ denote the set of recurrent configurations of~$S_n$ obtained by taking ${(k\eta)}^\circ$ for~${k \in \N}$.
	That is,
	\begin{equation*}
		\cC_n := \bigcup_{k \in \N} {(k\eta)}^{\circ}.
	\end{equation*}
	Further, for~${m \ge n}$, define the restriction
	\begin{equation*}
	  \cC_m\big|_{S_n} := \bigcup_{k \in \N} {\big(k\eta\big|_{S_n}\big)}^{\circ}.
	\end{equation*}
\end{defn}

By repeatedly adding~$\eta$, we obtain a cyclic group of recurrent configurations for any~$S_n$.
Thus ${\left|\cC_n\right|}$ is the maximum period of each vertex in~$S_n$.

\begin{prop} \label{prop:restrict}
	For a configuration~$\eta$ and a set of recurrent configurations~$\cC_n$ as in \ref{def:recurrent-config-cycle-set}, for all finite~$m$ and~$n$ such that ${m \geq n \geq 1}$,
	\begin{equation*}
		\cC_m\big|_{S_n} = \cC_n.
	\end{equation*}
\end{prop}
\begin{proof}
	To show~${\cC_m\big|_{S_n} \subseteq \cC_n}$, let ${\eta_n = {\big(k\eta\big|_{S_n}\big)}^{\circ}}$ for some~${k \in \N}$.
	By \ref{prop:Idf} we know that ${{(k\eta)}^{\circ} \oplus \idf = {(k\eta)}^{\circ}}$ and that each vertex topples exactly once during this stabilization.
	Carrying out this sequence of topplings and then toppling each~${v \in B_n}$ yields the configuration~${\eta_n + \idf^n}$.
	Since ${(k\eta)}^{\circ}$ is recurrent, after toppling each vertex once, ${\eta_n + \idf^n}$ stabilizes to~$\eta_n$.
	Thus ${\eta_n \oplus \idf^n = \eta_n}$, and $\eta_n$ is a recurrent configuration formed by adding~$\eta$ to itself.
	Therefore ${\eta_n \in \cC_n}$, as desired.

	For the reverse inclusion, again let ${\eta_n = {\big(k\eta\big|_{S_n}\big)}^{\circ}}$ and ${\eta_m = {\big(k\eta\big|_{S_m}\big)}^{\circ}}$.
	Let ${N = (N_1, N_2, \ldots)}$ represent a strictly increasing sequence of the number of times we can add $\eta$ to itself such that ${\eta_n \in \cC_n}$ occurs.
	We note that since $\eta_n$ is recurrent and part of a cyclic group, $N$ will be infinitely long, and thus ${N_k \to \infty}$ as~${k \to \infty}$.
	Consider~$N_k\eta$ in~$S_m$ for any given~${N_k \in N}$.
	Topple until stable every vertex up to but not including the boundary vertices~${v \in B_n}$.
	Since $\eta_n$ is recurrent, it is invariant of topplings of~${v \in B_n}$.
	Since $F_m$ is finite, we can thus choose~${N_k \in N}$ large enough so that the resulting stable configuration, $\eta_m$, is in $\cC_m$.
	However, since $N_k$ is from $N$, we have ${\eta_m\big|_{S_n} = \eta_n}$.
	Hence ${\eta_n \in \cC_m\big|_{S_n}}$\!, so ${\cC_n \subseteq \cC_m\big|_{S_n}}$\!.
\end{proof}

From \ref{prop:restrict}, since $\ds{\left|\cC_m\big|_{S_n}\right|\leq \left|\cC_m\right|}$, it follows that $\ds{\left|\cC_n\right| \leq \left|\cC_m\right|}$ for~${n \le m}$.
Given any~${v \in G}$, let $\ds{m = \min\{i \mid v \in S_i \}}$.
Then the maximum period of any~$v$ is given by~${|\cC_m|}$.

\begin{thm} \label{thm:division}
For all $n \leq m < \infty$, $|\cC_n|$ divides $|\cC_m|$.
\end{thm}

\begin{proof}
	Given any ${\eta_1, \eta_2 \in \cC_m}$, we define an equivalence relation~$\sim$ by declaring ${\eta_1 \sim \eta_2}$ if and only if
	\begin{equation*}
		\eta_1 \big|_{S_n} = \eta_2\big|_{S_n}.
	\end{equation*}
	Let $N_n$ denote the equivalence class of~$\idr$, which is the identity element of~$\cC_m$.
	Since the equivalence class of the identity element is always a normal subgroup of the original group, ${\cC_m \mod N_n}$ is well defined.
	Thus we have
	\begin{equation*}
		|\cC_m| = |N_n| \left|\eta_m\big|_{S_n}\right| = |N_n| |\cC_n|.
	\end{equation*}
	Hence ${|\cC_n|}$ divides~${|\cC_m|}$.
\end{proof}

We note that this statement is not empty as $SG$, $PG$, $HG$, and~$MG$ are examples of graphs that all satisfy the above properties when adding chips to a central point of the fractal graph.

Finally, we present two conjectures based on observations of these periodic properties.

\begin{conj}
	Let $K$ denote any of the following fractal graphs.
	Let $\ds{m = \min\{i \mid v \in K_i \}}$.
	The maximum period of any vertex~$v$ is given by:
	\begin{itemize}
		\item $SG$: $4 (3^m)$,
		\item $HG$: $2(3^m)$,
		\item $PG$: $6(5^m)$,
		\item $MG$: $6(7^m)$.
		\item $SGC$: $6(7^m)$.
	\end{itemize}
\end{conj}

\begin{conj}\label{conj:period-equals-cardinality}
	Let ${v \in SG}$ and let $\ds{m = \min\{i \mid v \in SG_i \}}$. Then the period of~$v$ is given by $\left|\cC_m\right|$.
\end{conj}

We believe \ref{conj:period-equals-cardinality} holds for~$SG$ since there are no vertices that always have the same value once a recurrent configuration is reached.
However, for~$SGC$ and~$HG$, there are vertices that never change their value in all recurrent configurations.

%==============================================================================

\section{Structure of the Sandpile Group}\label{sec:group-structure}

We can gain insight on the structure of the Sandpile group by using the Smith Normal Form of the reduced graph Laplacian (see \cite{Lorenzini}).
The Smith Normal Form is a diagonal integer matrix in which the nonzero diagonal entries give the presentation of the Sandpile group.
For example, suppose the following matrix were the Smith Normal Form of a graph's reduced Laplacian:
\begin{equation*}
	\begin{bmatrix}
		5 & 0 & 0\\
		0 & 10 & 0\\
		0 & 0 & 0
	\end{bmatrix}\!.
\end{equation*}
Then the graph's Sandpile group would be given by~${\Z_5 \times \Z_{10}}$.

For the fractals we consider, we need to decide how the sink vertex should be attached to the graph approximations and how this vertex interacts with the graph's boundary vertices.
We experiment with two boundary conditions:
\begin{description}
	\item[Normal] The boundary points of the fractal are each connected to a separate sink vertex. In $SG$ we include two edges from each boundary point to the sink; in $SGC$ was include one edge from each boundary point to the sink.
	\item[Sinked] The boundary points of the fractal are collapsed to form a single sink vertex.
\end{description}
The normal boundary condition has been used in other work, such as \cite{daerden}.
However, by considering instead the sinked boundary condition, we are able to conjecture a closed-form formula for the order of the Sandpile group of~$SG$.

The order of Sandpile group associated to a directed graph is given by the number of spanning trees rooted at a fixed vertex (\cite{Levine, Lionel}).
In~$SG$ we calculate the number of spanning trees of the sequence of graph approximations to~$SG$ to obtain the following conjecture.
\begin{conj}
	Using the sinked boundary condition, the order of the Sandpile group of~$SG$ is given by the product ${2^{f_n} \, 3^{g_n} \, 5^{h_n}}$, where $f_n$, $g_n$, and $h_n$ are
	\begin{equation*}
		f_n = \frac{1}{2} (3^n - 1), \qquad g_n = \frac{1}{4} (3^{n+1} - 6n - 3), \qquad h_n = \frac{1}{4} (3^n + 6n - 1).
	\end{equation*}
\end{conj}
This conjecture is based on results from \cite{Anema}, which uses a modified version of Kirchhoff's matrix tree theorem to calculate the number of spanning trees as a product of the nonzero eigenvalues of the graph's probabilistic Laplacian.

In the following tables, we identify the structure of the Sandpile group for several levels of the Sierpinski Gasket and for its cell graph.

\setlength{\tabcolsep}{0.25em}
\begin{center}
	\footnotesize
	\begin{tabular}{|c l | c l|}
		\hline
		\multicolumn{4}{|c|}{Normal Boundary Condition}\\
		\hline
		$SG_0$ & $\Z_2 \times {(\Z_5)}^2$ & $SGC_1$ & ${(\Z_4)}^2$\\
    $SG_1$ & ${(\Z_2)}^2 \times {(\Z_{19})}^2$ & $SGC_2$ & $\Z_5 \times {(\Z_{17})}^2$\\
		$SG_2$ & ${(\Z_2)}^5 \times {(\Z_3)}^3 \times \Z_5 \times {(\Z_7)}^2 \times {(\Z_{11})}^2$ & $SGC_3$ & ${(\Z_3)}^3 \times {(\Z_4)}^2 \times {(\Z_5)}^3 \times {(\Z_{19})}^2 \times \Z_{25}$\\
		&& $SGC_4$ & ${(\Z_3)}^9 \times {(\Z_5)}^9 \times {(\Z_9)}^3 \times {(\Z_{25})}^3 \times \Z_{125} \times {(\Z_{353})}^2$\\
		\hline
		\hline
		\multicolumn{4}{|c|}{Sinked Boundary Condition}\\
		\hline
		$SG_1$ & $\Z_2 \times {(\Z_5)}^2$ & $SGC_2$ & $\Z_5 \times {(\Z_8)}^2$\\
    $SG_2$ & ${(\Z_2)}^4 \times {(\Z_3)}^3 \times \Z_5 \times {(\Z_{25})}^2$ & $SGC_3$ & ${(\Z_3)}^3 \times {(\Z_5)}^3 \times \Z_{25} \times {(\Z_{49})}^2$\\
		$SG_3$ & ${(\Z_2)}^{13} \times {(\Z_3)}^9 \times {(\Z_5)}^3 \times {(\Z_9)}^3 \times \Z_{25} \times {(\Z_{125})}^2$ & $SGC_4$ & \parbox[t]{6cm}{${(\Z_3)}^9 \times {(\Z_5)}^9 \times {(\Z_9)}^3 \times {(\Z_{16})}^2 \times {(\Z_{17})}^2$\\$\times {(\Z_{25})}^3 \times \Z_{125}$}\\
		\hline
	\end{tabular}
\end{center}

%-------------------------------------------------------------------------------

\section{Identities}\label{sec:identities}

There are multiple works that deal with characterizing the identity element of the Sandpile group, see \cite{DharAlgebraicAspects, Caracciolo, BorgneIdentity, DartoisIdentity}.
In this section we find the identity configuration for~$SGC$ and provide conjectures for the identity element of several other fractals.

\begin{figure}[!ht]
	\centering
	\begin{subfigure}{0.40\textwidth}
		\includegraphics[scale = 0.40]{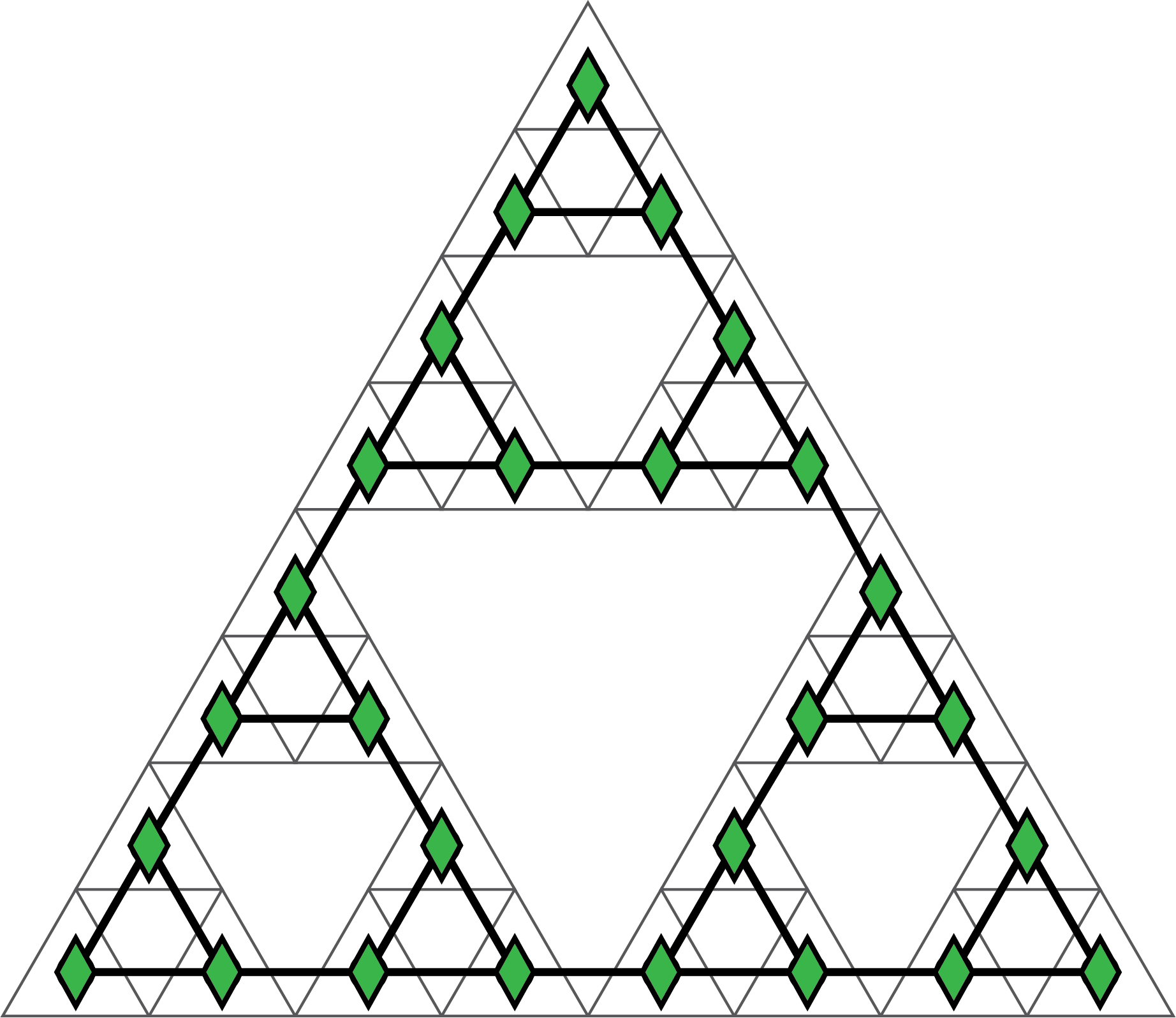}
		\caption{}
		\label{fig:SGCell_Id}
	\end{subfigure}
	\begin{subfigure}{0.36\textwidth}
		\includegraphics[scale = 0.36]{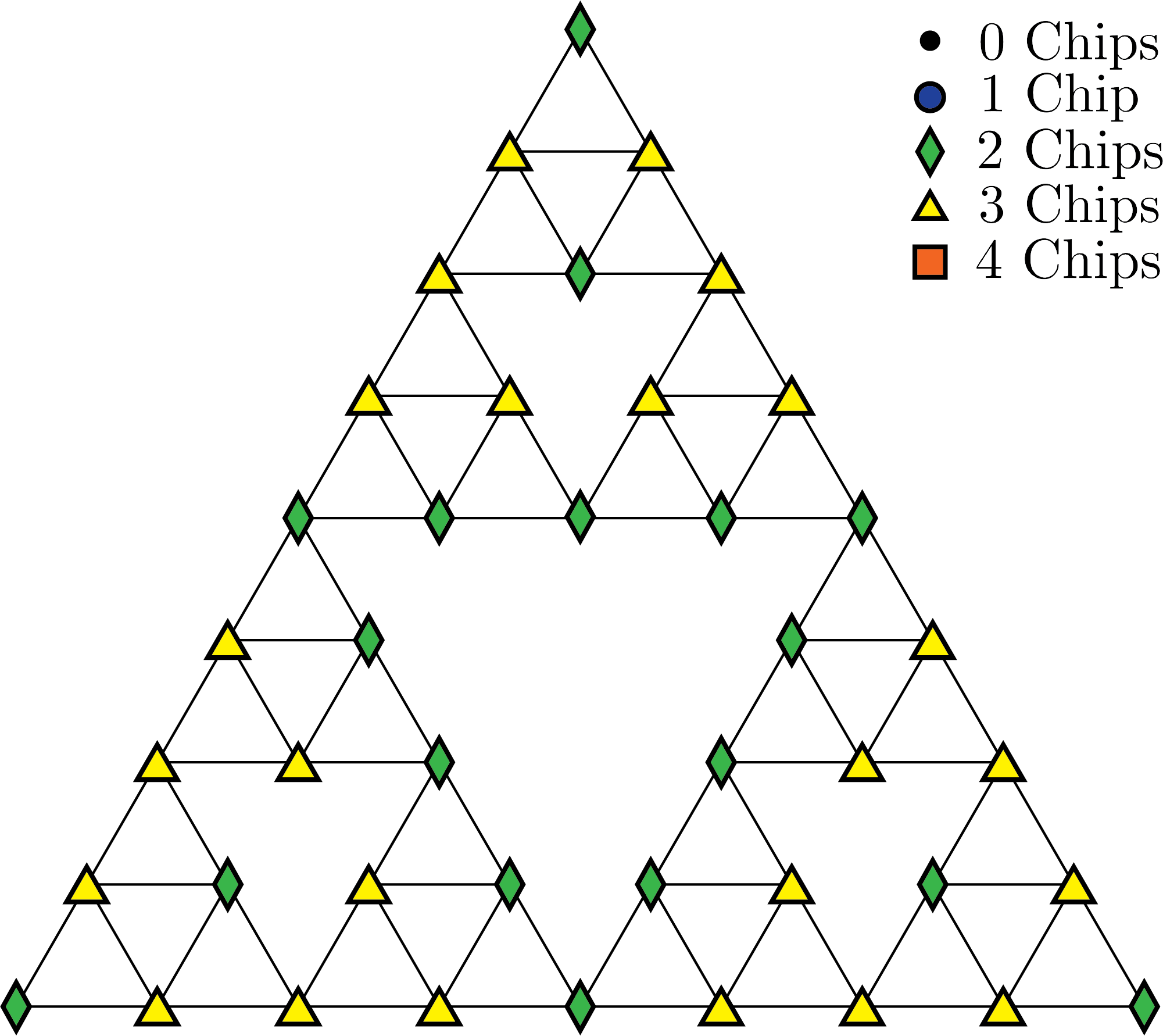}
		\caption{}
		\label{fig:SG_Id}
	\end{subfigure}
	\caption{Identities of~(a)~$SGC$ and (b)~$SG$.}
\end{figure}

The identity elements of these fractals exhibit clear patterns (we use the normal boundary condition in analyzing these identity configurations).
$SGC$'s identity has $2$~chips on each vertex.
$SG$'s identity has $2$~chips on the ``inside'' vertices of every subgraph $2$~chips on the boundary vertices, and $3$~chips on all other vertices.
Next we prove that $SGC$'s identity element is indeed this $2$-chip configuration.

\begin{figure}[!ht]
	\centering
	\includegraphics[scale=0.4]{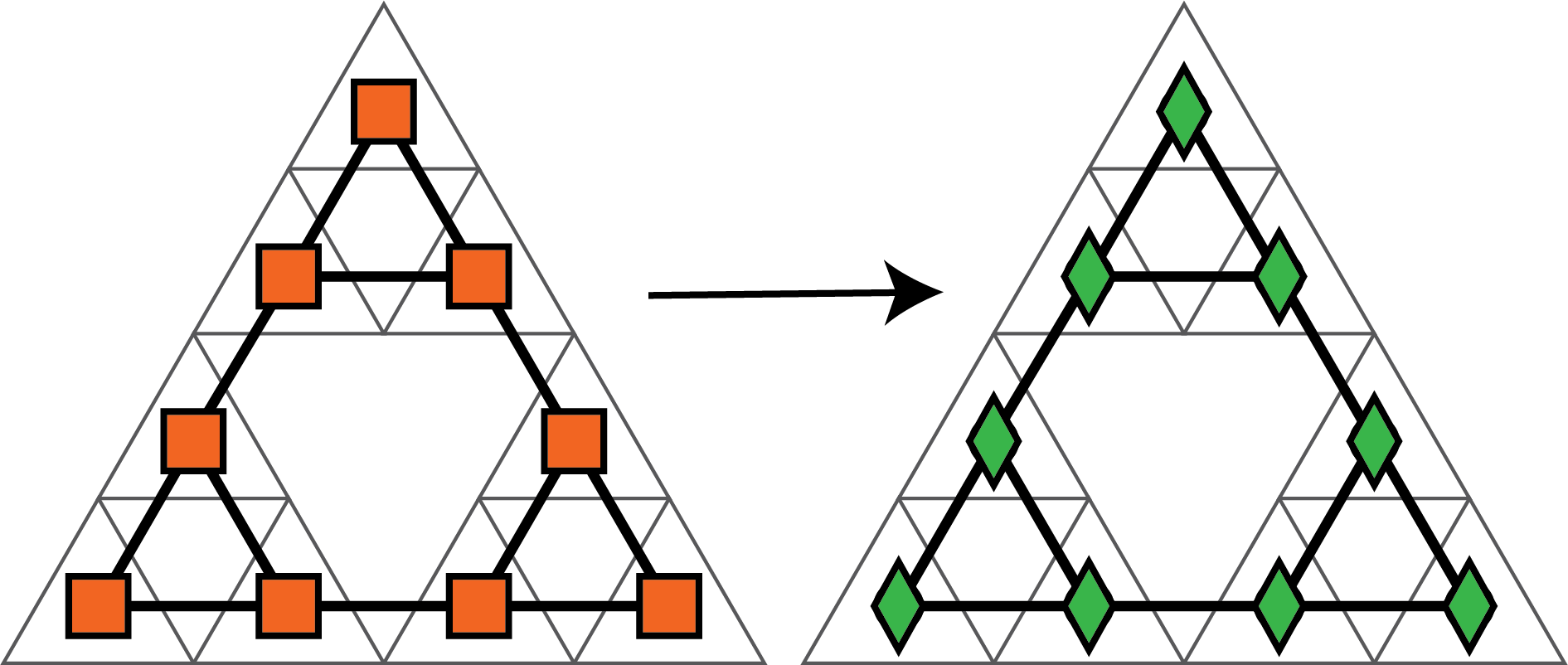}
	\caption{A $4$-chip configuration stabilizes to a $2$-chip configuration.}
	\label{fig:sketch0}
\end{figure}

\begin{thm}
  \label{thm:identity}
  For all levels~${n > 0}$, the identity of~$SGC$'s Sandpile group with the normal boundary condition is the configuration with $2$~chips on each vertex.
\end{thm}

Letting~$e$ denote the configuration with $2$~chips on each vertex, we show that ${e \oplus e = e}$.
This process entails stacking two $2$-chip configurations on top of each other, resulting in $4$~chips on each vertex, and then showing that this configuration stabilizes to the $2$-chip configuration.
Since we can topple chips in any order, we just need to find an order of topplings that is simple enough for us to explain.

\begin{figure}[!h]
	\centering
	\includegraphics[scale=0.5]{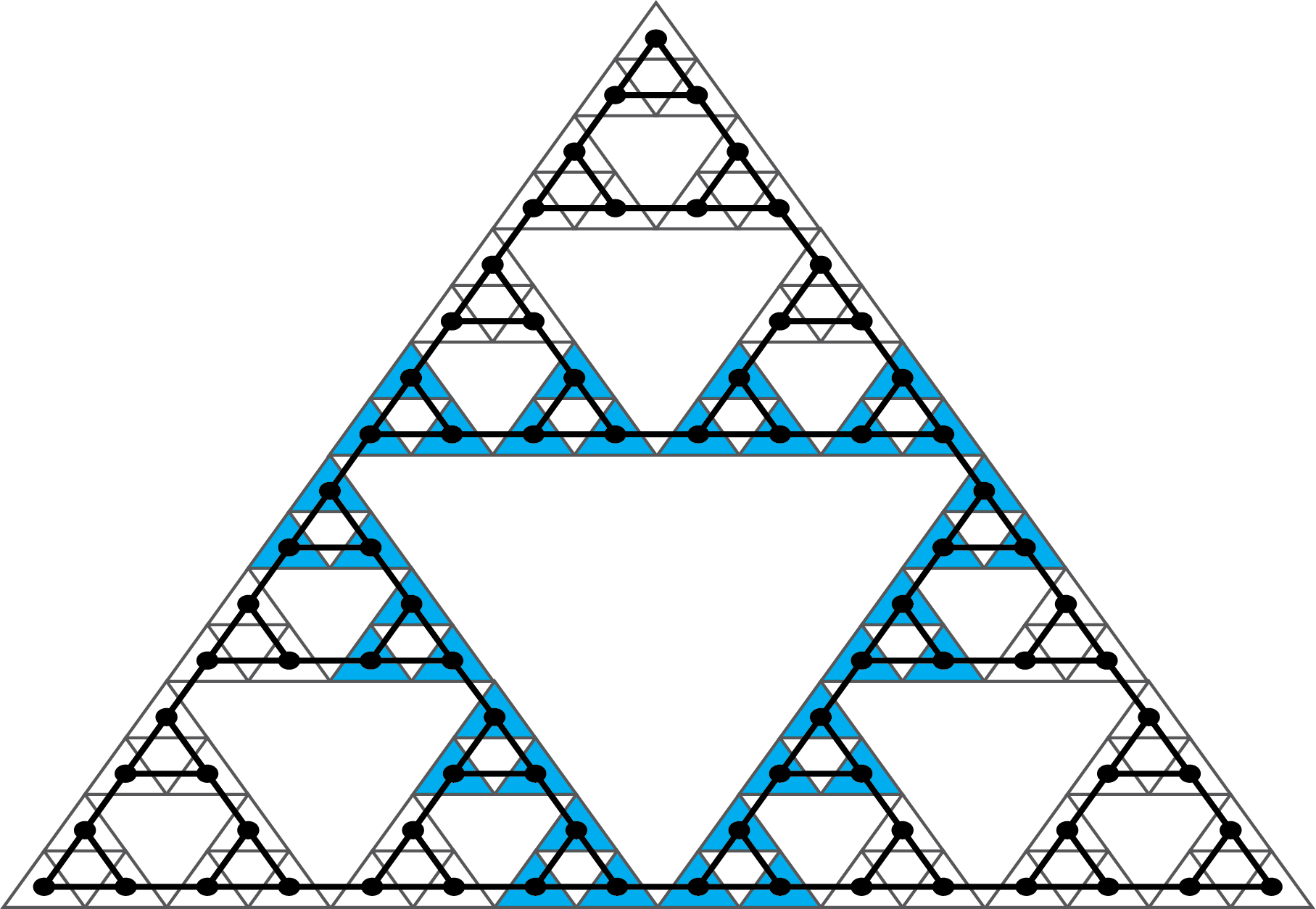}
	\caption{The inner ring.}
	\label{fig:sketch1}
\end{figure}

Before beginning the proof, we first examine the structure of the cell graph to understand how the subgraphs interact with each other.
A cell graph contains an inner ring of level $1$~triangles.
We'll need a property concerning the behavior of chips on this ring during the stabilization process.

\begin{lemma}[Ring of Triangles]
	\label{lemma:Ring_of_Triangles}
	Consider a chain of~$t$ level~$1$ triangles in which each outer vertex has one edge to the sink, as in \ref{fig:sketch2}.
	Place $2$~chips on every inner vertex of each triangle and $3$~chips on the outer vertex of each triangle.
	The stable configuration consists of $2$~chips on each vertex.
	Further, every corner vertex topples exactly once, resulting in $t$~chips toppling into the sink.

	Moreover, we can start with any number of chips~${m \geq 3}$ on each of the corner vertices.
	The stable configuration has $2$~chips on each vertex.
	Every corner vertex topples exactly ${m - 2}$~times, and ${t(m - 2)}$~chips in total topple into the sink.
\end{lemma}

\begin{figure}[H]
	\centering
	\includegraphics[scale=0.5]{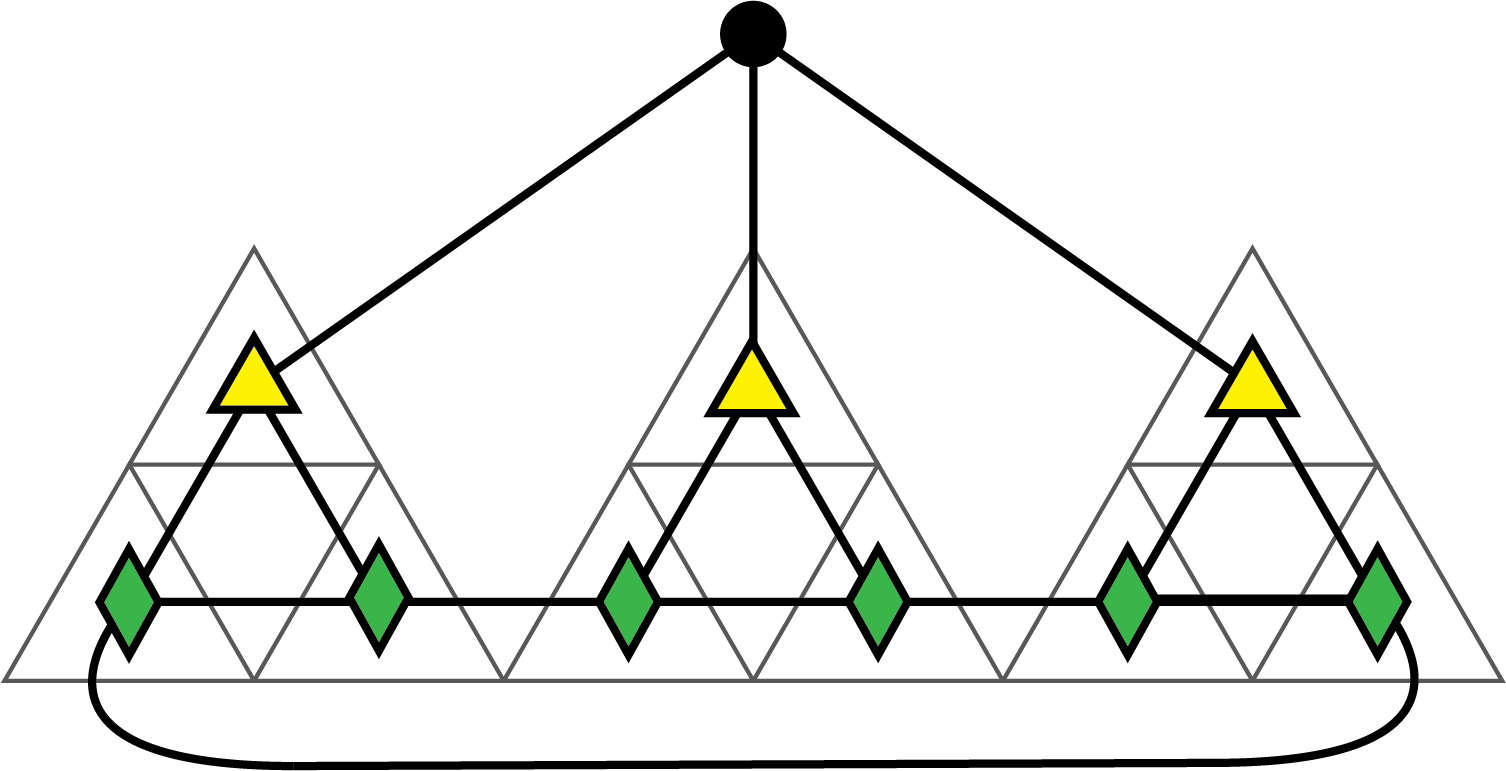}
	\caption{3 chips on the corner vertices, 2 chips on the inside vertices.}
	\label{fig:sketch2}
\end{figure}

\begin{proof}
	First topple all of the corner vertices, losing one chip each to the sink and leaving us with $0$ chips on these corners and $3$ chips on all of the other vertices.
	Now topple these inner vertices, producing the desired $2$ chip configuration.

	Now suppose that we start with $m$~chips on each corner vertex instead of three.
	We can topple in any order, so we can apply the $3$~chip case iteratively.
	Each vertex topples exactly ${m - 2}$ times, so ${t(m - 2)}$~chips in total topple into the sink.
\end{proof}

Next we want to identify the behavior of a larger triangle, as each level~${n + 1}$ graph contains an inner ring of six level~${n - 1}$ graphs (see \ref{fig:sketch4}).
We show inductively that the same result holds for triangles of level~${n \geq 1}$.
That is, if we start with $3$~chips on one corner and $2$~chips on all other vertices, we end up with $2$~chips everywhere except for the other corners, each of which have $1$~chip.
\begin{figure}[H]
	\centering
	\includegraphics[scale=0.5]{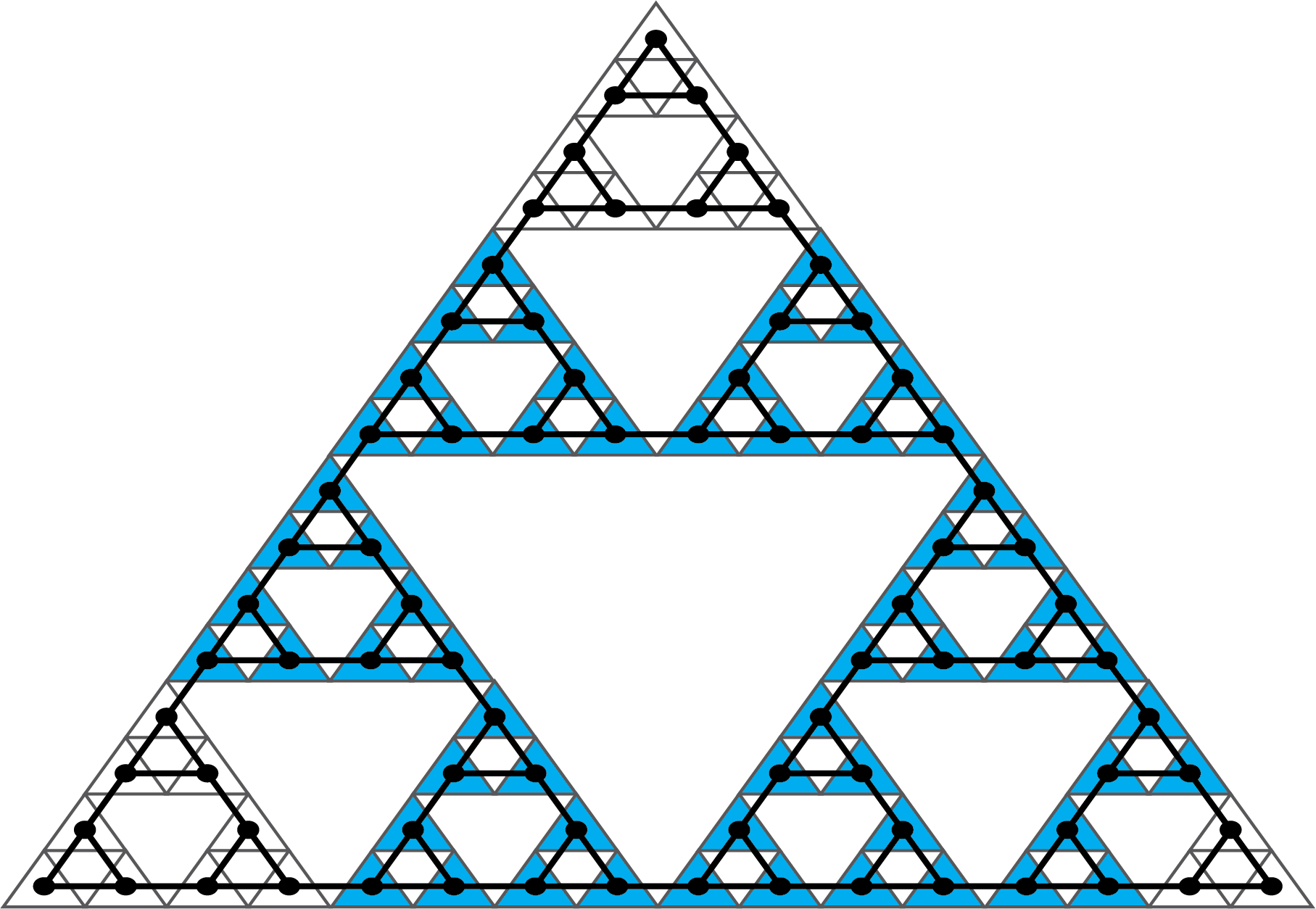}
	\caption{Six inner level~$n$ subgraphs.}
	\label{fig:sketch4}
\end{figure}

\begin{prop}[A Larger Triangle]
	\label{prop:A_Larger_Triangle}
	\ref{lemma:Ring_of_Triangles} holds for all cell graphs of level~${n \geq 1}$.
	That is, consider a chain of~$t$ level~${n \geq 1}$ cell graphs in which each corner vertex has a single edge to the sink.
	Place ${m \geq 3}$~chips on each of these corner vertices and $2$~chips on each of the other vertices.
	The stable configuration then consists of $2$~chips on each vertex, and the odometer of each outer vertex is~${m - 2}$.
\end{prop}
\begin{proof}
	\begin{figure}[H]
		\centering
		\includegraphics[scale=0.55]{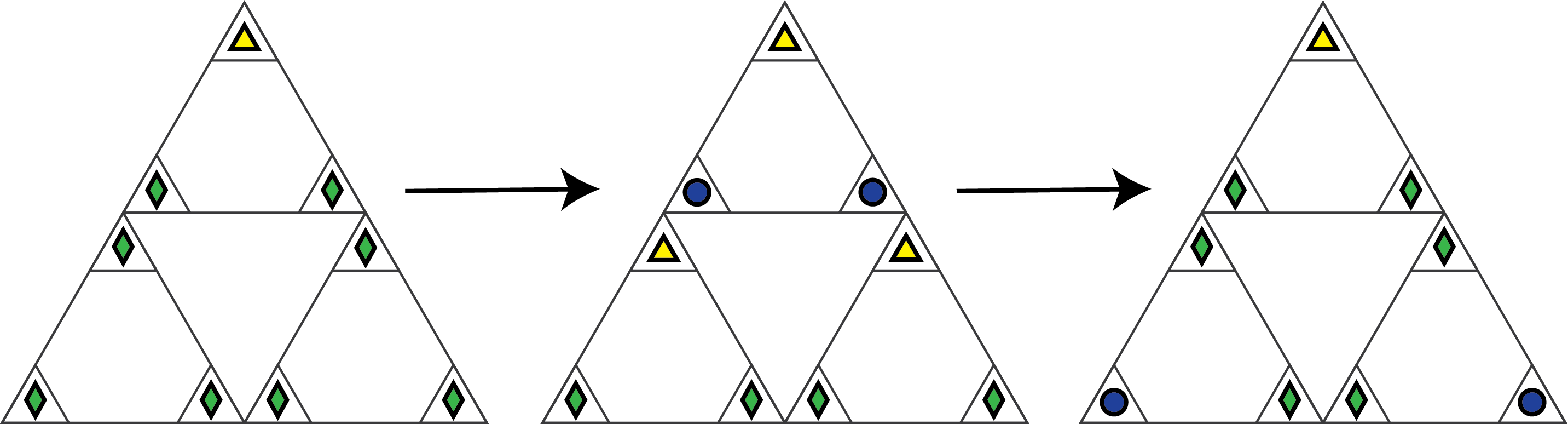}
		\caption{Toppling from 3 chips on one corner.}
		\label{fig:sketch5}
	\end{figure}
	Assume that the result holds for a level~$n$ cell graph starting with $3$~chips on the outer vertices.
	Apply this case to the top third of a level~${n + 1}$ cell graph to leave $2$~chips on every vertex of the top subgraph except the corners, which each have $1$~chip.
	Then apply the level~$n$ case again to the two lower thirds to obtain the desired result, noting that when the top vertex of each of these subgraphs topple, one chip is returned to a corner of the top third (see \ref{fig:sketch5}).
	The original claim now follows using the same argument as in \ref{lemma:Ring_of_Triangles}.
\end{proof}

Now we move on to the main proof of \ref{thm:identity}: that the identity of a level~$n$ cell graph is the configuration with $2$~chips on every vertex.

\begin{proof}
	The first few levels are easily verified by hand.
	For the induction case we construct a sequence of topplings that stabilizes the $4$-chip configuration.
	We also need to keep track of the odometer functions of the boundary vertices.
	Assume that for all~${k \leq n}$, the identity element of the level~$k$ cell graph consists of $2$~chips on each vertex.
	Further assume that when the configuration stabilizes from $4$~chips everywhere, the odometer of each boundary vertex is~${2 \cdot 3^{k-1}}$.
	Note that the vertex odometers can be obtained by counting the number of vertices, which is~$3^k$, multiplying by~$2$ for the $2$~chips per vertex that must topple, and dividing by~$3$ for symmetry.

	\begin{figure}[H]
		\centering
		\includegraphics[scale=0.4]{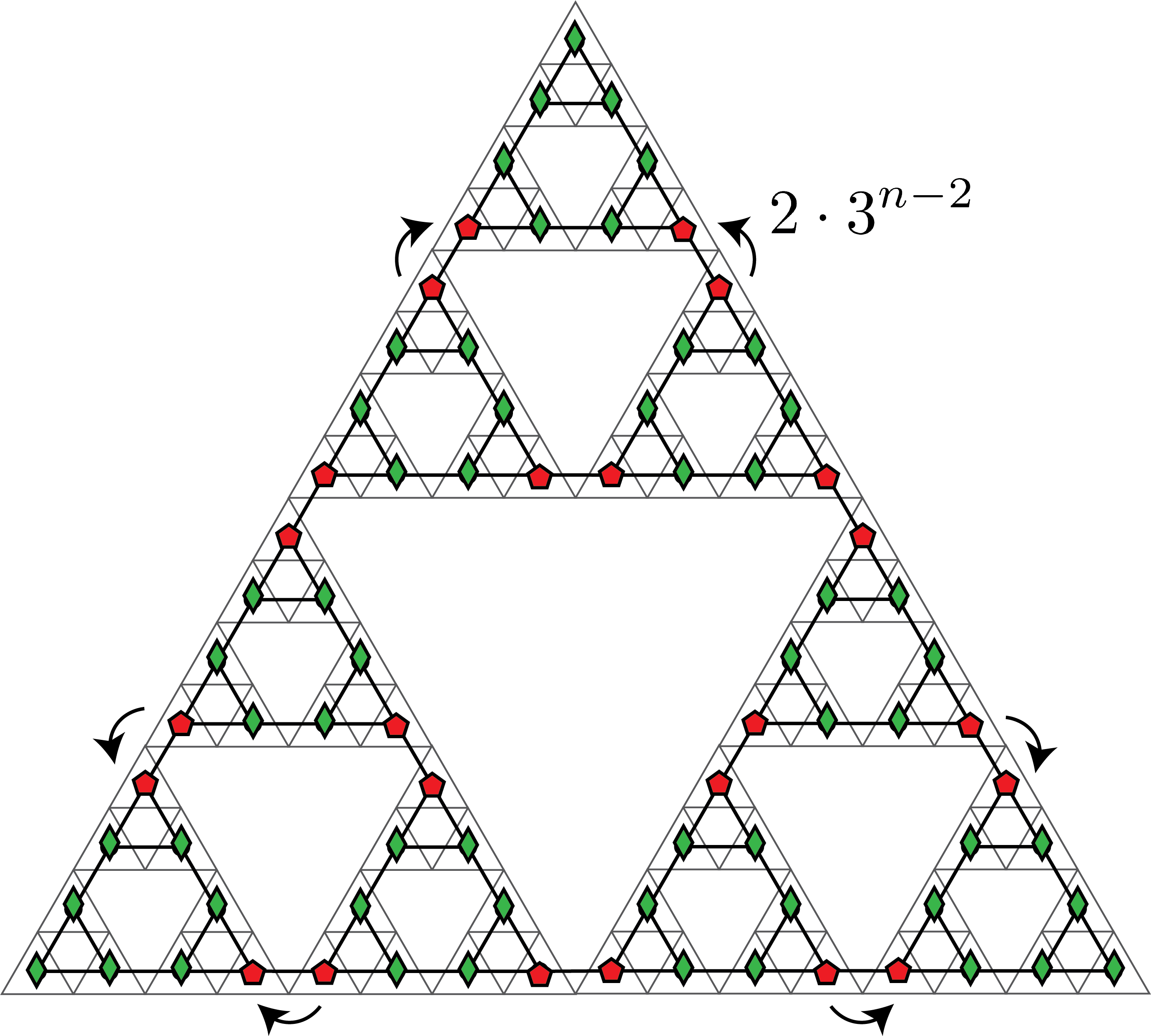}
		\caption{
			Neighboring vertices after subgraph stabilization, where each pentagon represents ${2 + 2\cdot3^{n-2}}$~chips.
			Applying \ref{prop:A_Larger_Triangle}, a stack of ${2\cdot3^{n-2}}$ chips are moved to the outer level~${n-1}$ subgraphs.
		}
		\label{fig:sketch8}
	\end{figure}

	Consider a level~${n + 1}$ subgraph with $4$~chips on each vertex.
	We break down the level~${n + 1}$ graph into nine level~${n - 1}$ subgraphs, the identity of each we have assumed to be $2$~chips on every vertex.
	Stabilize each of these level~${n - 1}$ subgraphs individually.
	Ignore any chips deposited on the corners of each subgraph by its neighbors, just stabilize the initial chips on each subgraph.
	Since each of these subgraphs has $4$~chips on every vertex, each stabilizes to $2$~chips everywhere and loses ${2 \cdot 3^{(n-1)-1} = 2 \cdot 3^{n-2}}$~chips from each of its corner vertices.

	Our graph now has an inner ring of six level~${n - 1}$ subgraphs, each of which has $2$~chips on every vertex except for its corners, which each have ${2 + 2 \cdot 3^{n-2}}$~chips (see \ref{fig:sketch8}).
	We want to show that this inner ring stabilizes to $2$~chips everywhere.
	First we apply \ref{prop:A_Larger_Triangle} to clear the chips from the boundary of this ring, removing ${2 \cdot 3^{n-2}}$~chips from each such vertex.

	We continue clearing chips off of the six inner triangles, but we note that by the symmetry of these graphs it is sufficient to show that a ring of three such triangles stabilizes to the $2$-chip configuration.
	We assumed that the identity of the level~$n$ graph is the $2$-chip configuration, and since we can topple chips in any order, we have that this inner ring stabilizes to the desired $2$-chip configuration.
	Thus we have $2$~chips on every vertex of the level~${n + 1}$ graph, except for the vertices neighboring this inner ring, which each have ${2 + 8 \cdot 3^{n-2}}$~chips (see \ref{fig:sketch10}).

	\begin{figure}[H]
		\centering
		\includegraphics[scale=0.62]{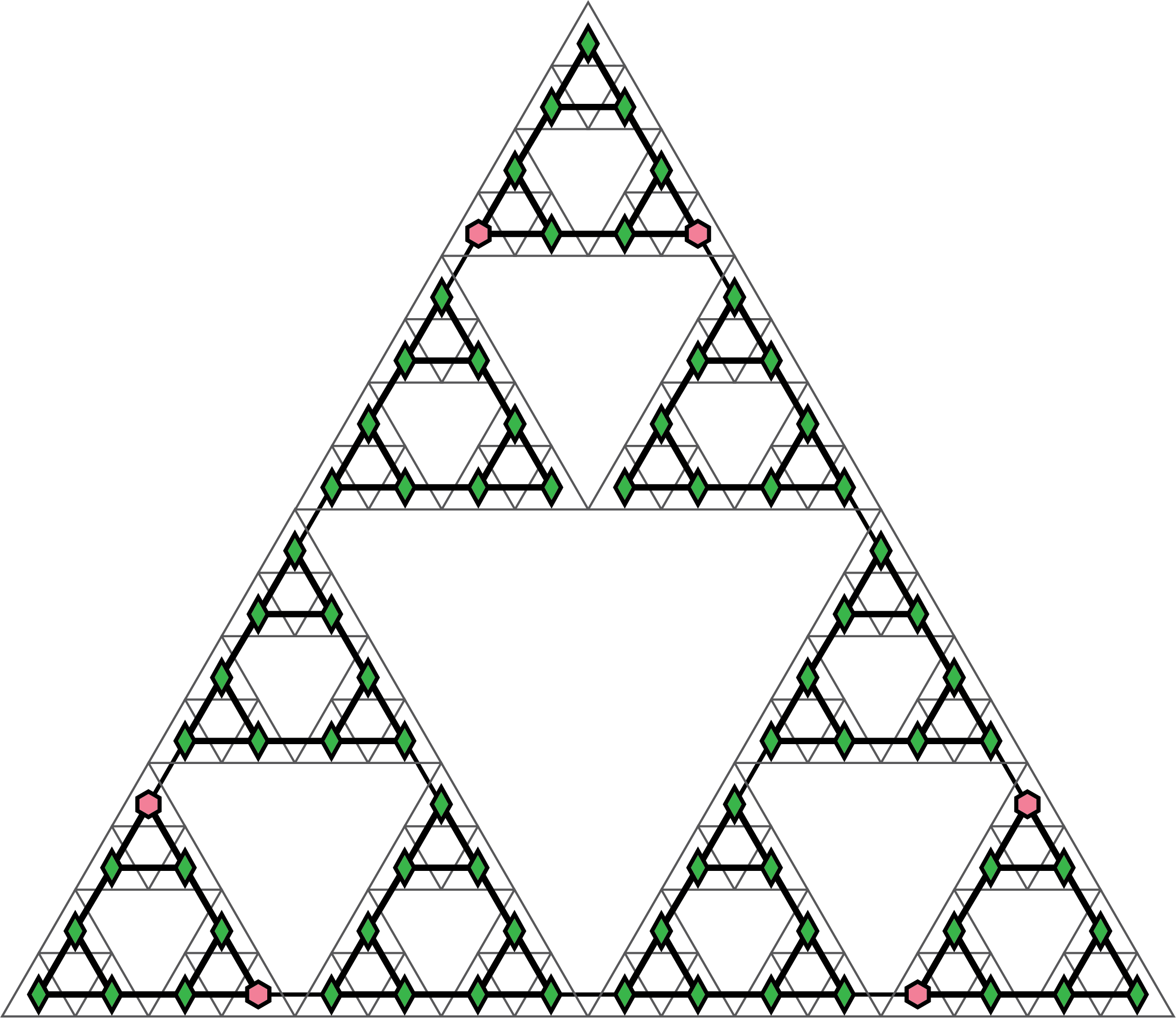}
		\caption{
			The inner level~$n$ cells with $2$~chips on each vertex.
			The hexagons indicate a stack of ${2 + 8\cdot3^{n-2}}$~chips.
		}
		\label{fig:sketch10}
	\end{figure}

	Consider toppling each vertex that has excess chips in \ref{fig:sketch10} exactly once. Then each triangle in the inner ring would have $3$~chips on its outer boundary vertex, and stabilizing just this inner ring produces~$2$ chips on every vertex of the inner ring by \ref{prop:A_Larger_Triangle}.
	Thus we can ignore this inner ring when toppling the remaining vertices.
	Then by our strong induction hypothesis, if these vertices had~${2 + 2\cdot3^{n-2}}$ chips, then we would end up with two chips on all vertices.
	But since $8$ is a multiple of~$2$, we simply can repeat this toppling process four times to end up with the same configuration.
	Hence we obtain the desired configuration of a level~${n + 1}$ triangle with two chips on each vertex.
\end{proof}

Finally, we computationally find the identity elements for low level graph approximations for several other fractals.
Images of the identity configurations for~$MG_2$, $MG_3$, $PG_2$, and $HG_2$ follow.
For~$MG$ we use two edges from each boundary vertex to the sink, the same boundary condition used for~$SG$.
For~$PG$ and~$HG$ we connect the boundary vertices to the sink using a number of edges such that the degree of each boundary vertex is the same as that of the junction vertices: each boundary vertex of~$PG$ has four edges to the sink, and each boundary vertex of~$HG$ has five edges to the sink.

\begin{figure}[H]
	\centering
	\includegraphics[scale=0.45]{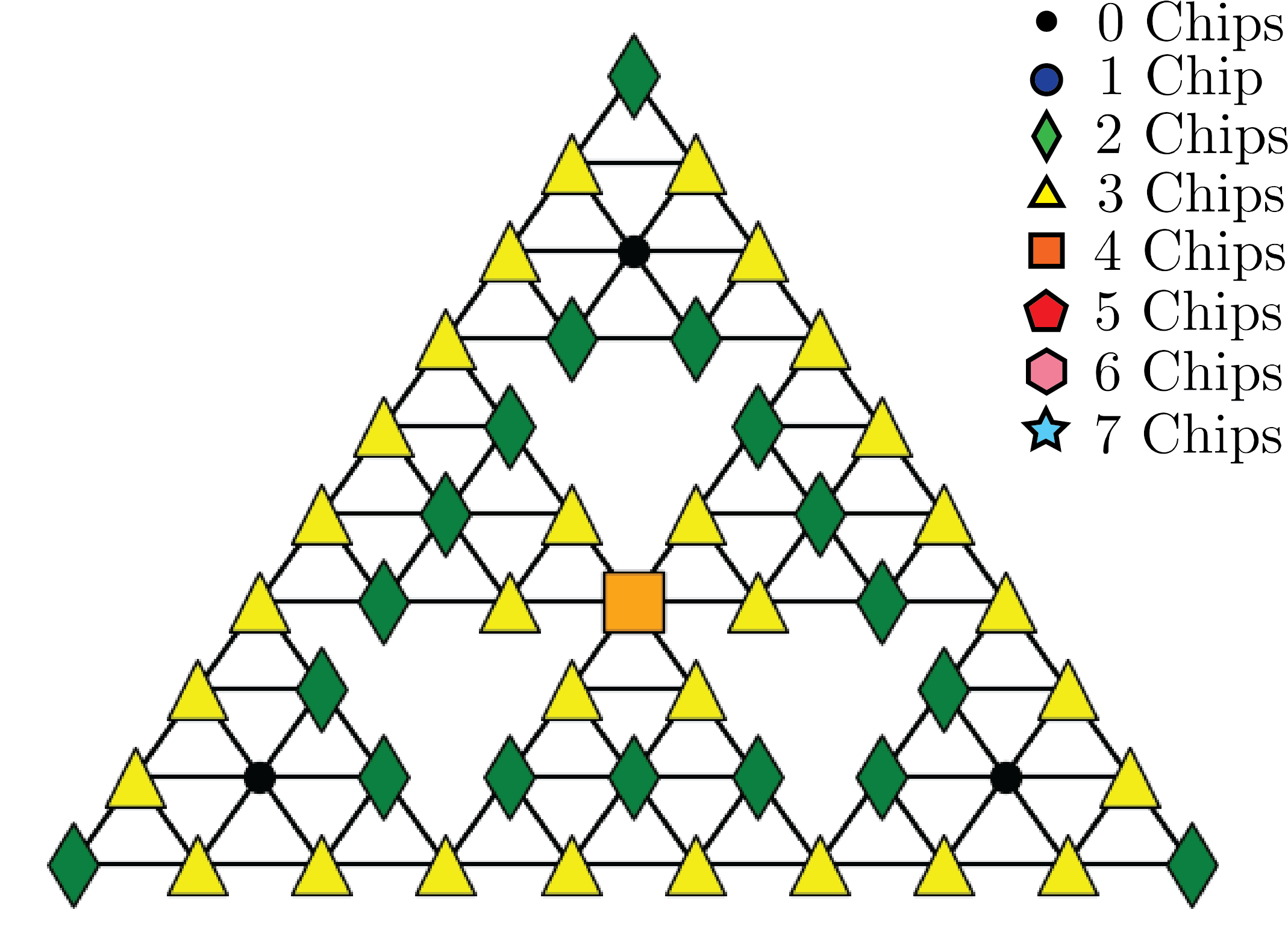}
	\caption{Identity element for~$MG$ level~$2$.}
	\label{fig:MG_Id2}
\end{figure}

\begin{figure}[H]
	\centering
	\includegraphics[scale=0.28]{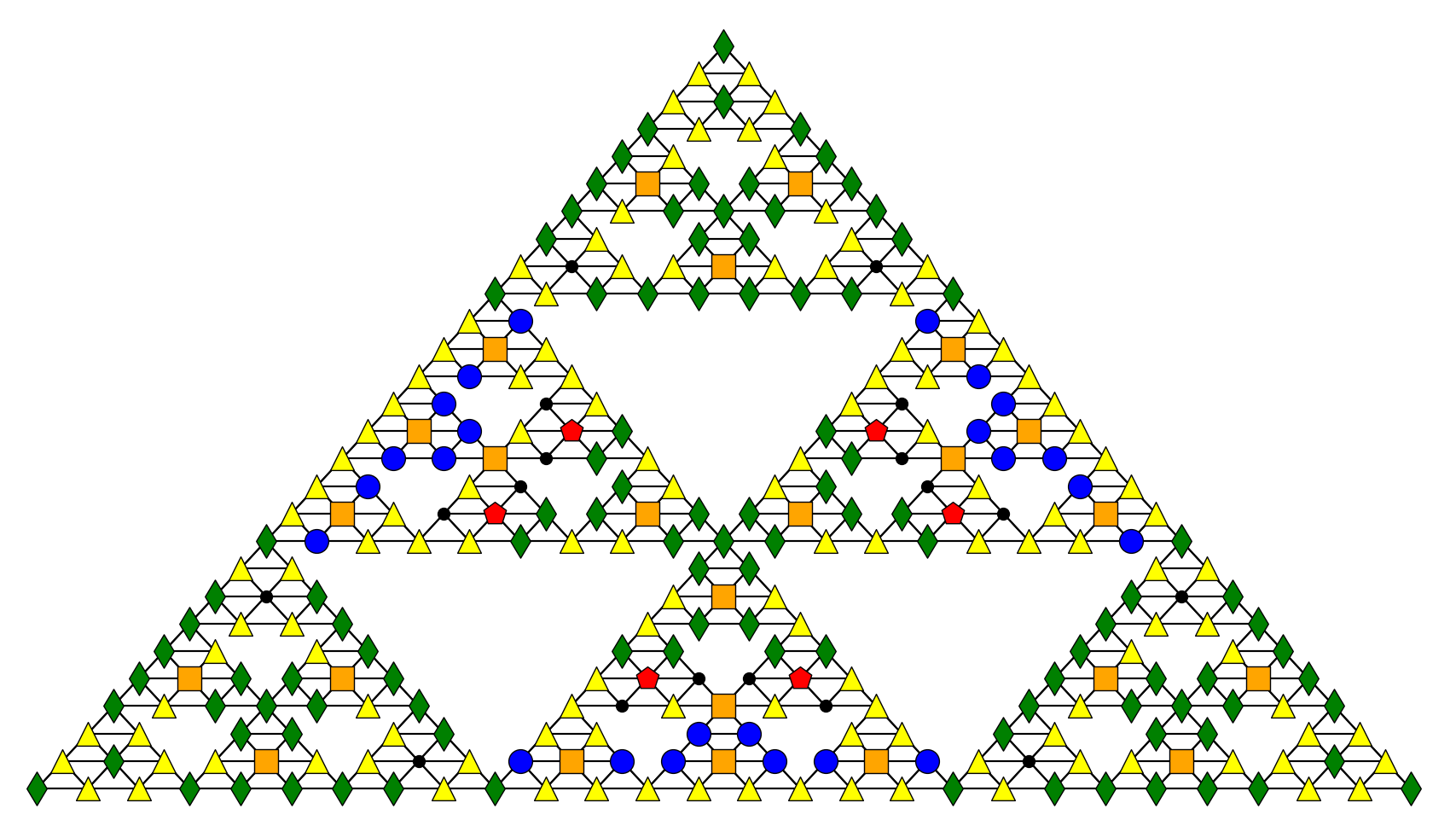}
	\caption{Identity element for~$MG$ level~$3$.}
	\label{fig:MG_Id}
\end{figure}

\begin{figure}[H]
	\centering
	\includegraphics[scale=0.3]{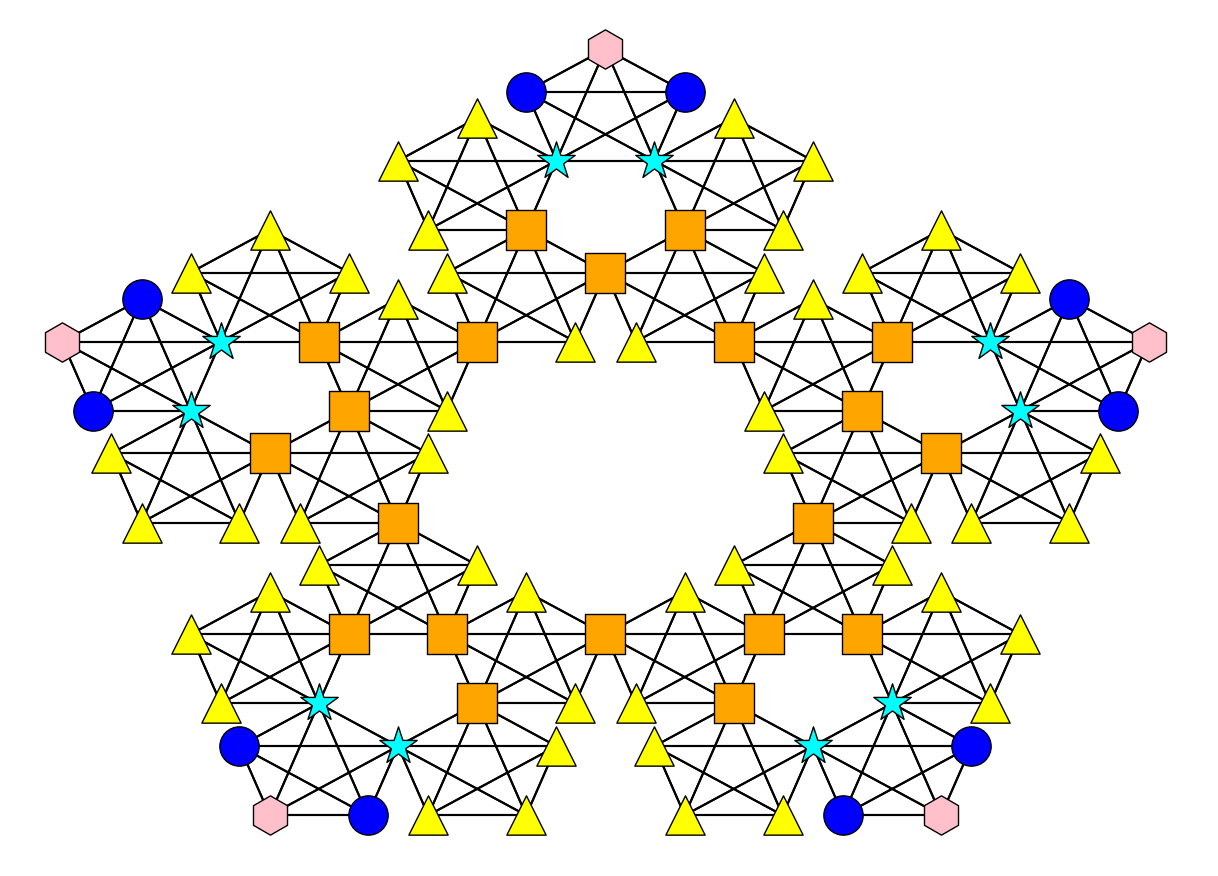}
	\caption{Identity element for~$PG$ level~$2$.}
	\label{fig:PG_Id}
\end{figure}

\begin{figure}[H]
	\centering
	\includegraphics[scale=0.2]{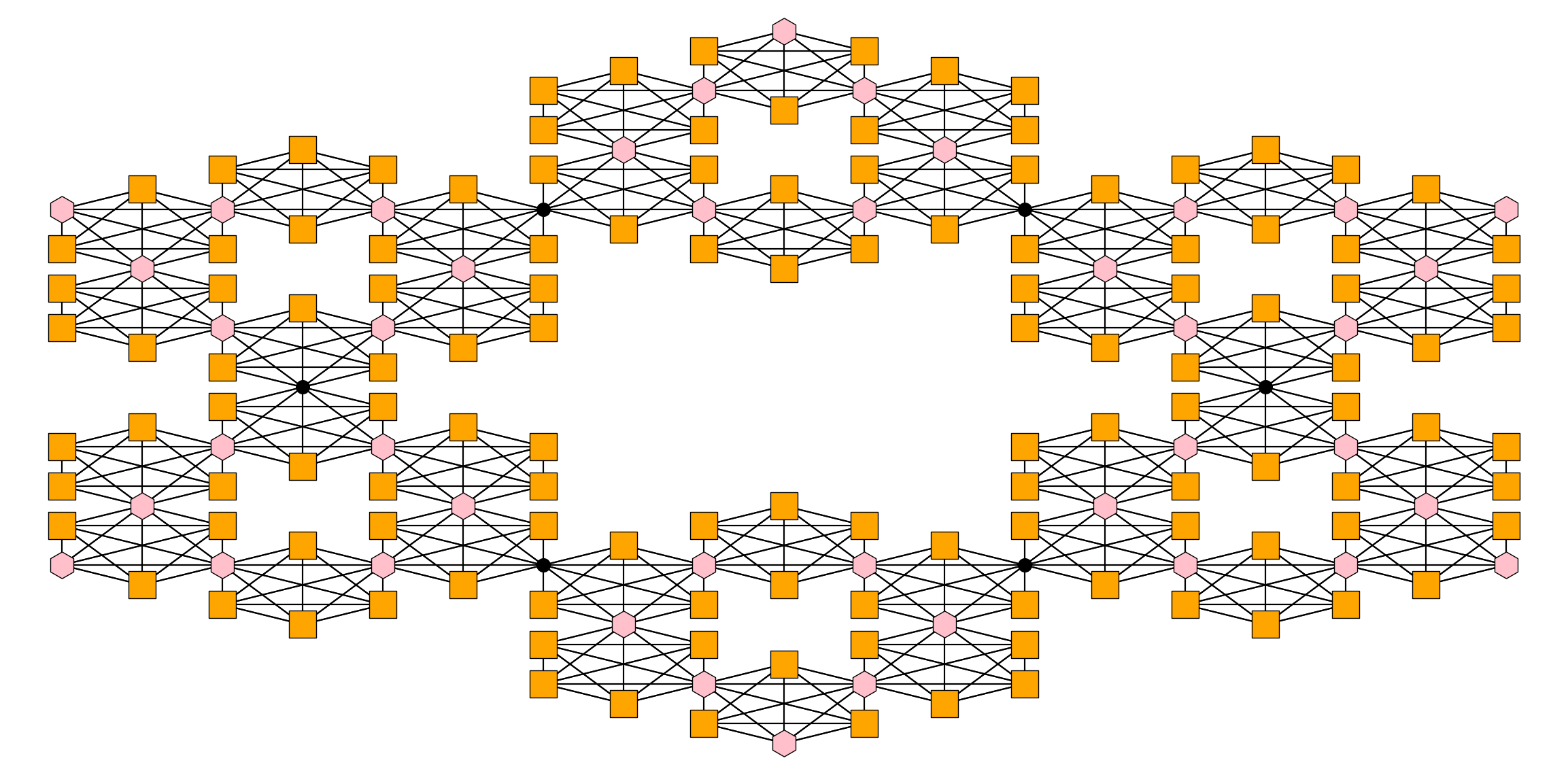}
	\caption{Identity element for~$HG$ level~$2$.}
	\label{fig:HG_Id}
\end{figure}

\section{Acknowledgements}
Ilse, Sam, and Rafael would like to acknowledge support by the National Science Foundation through the Research Experiences for Undergraduates Programs at Cornell, grant DMS-1156350.
Robert Strichartz would like to acknowledge support by the National Science Foundation, grant DMS-1162045.
We are all grateful to Lionel Levine for his many useful suggestions.

\nocite{*}
\bibliographystyle{ieeetr}
\bibliography{Sources}

\smallskip
\small
Websites: \url{http://www.math.cornell.edu/\~skayf/} and \url{http://www.math.cornell.edu/\~twestura/}

\end{document}